\documentclass[11pt]{amsart}
\usepackage{amssymb}
\usepackage{amsmath}
\usepackage{hyperref}

\newtheorem{theorem}{Theorem}[section]
\newtheorem{lemma}[theorem]{Lemma}
\newtheorem{hypothesis}[theorem]{Hypothesis}
\newtheorem{condition}[theorem]{Condition}
\newtheorem{proposition}[theorem]{Proposition}
\newtheorem{definition}[theorem]{Definition}
\newtheorem{corollary}[theorem]{Corollary}
\newtheorem{conjecture}[theorem]{Conjecture}

\theoremstyle{remark}
\newtheorem{remark}[theorem]{Remark}
\newtheorem{example}[theorem]{Example}
\numberwithin{equation}{section}

\newcommand{\pair}[2]{\langle#1,#2\rangle}
\newcommand{\floor}[1]{\lfloor #1\rfloor}

\newcommand{\mhp}{\vskip -.3cm}
\newcommand{\qp}{\vskip .12cm}
\newcommand{\hp}{\vskip .3cm}
\newcommand{\p}{\vskip .5cm}

\newcommand{\Sym}{\text{Sym}}
\newcommand{\Z}{\mathbb{Z}}
\newcommand{\Q}{\mathbb{Q}}
\newcommand{\Ql}{\mathbb{Q}_{\ell}}
\newcommand{\CO}{\mathcal{O}}
\newcommand{\til}{\tilde}
\newcommand{\bsl}{\backslash}
\newcommand{\ra}{\rightarrow}
\newcommand{\xra}{\xrightarrow}
\newcommand{\inn}{\langle\cdot,\cdot\rangle}

\newcommand{\bx}{\mathbf{x}}
\newcommand{\se}{\mathsf{e}}

\newcommand{\bH}{\mathbf{H}}
\newcommand{\tS}{\tilde{S}}
\newcommand{\tB}{\tilde{B}}
\newcommand{\tG}{\tilde{G}}

\newcommand{\Stab}{\text{Stab}}
\newcommand{\Ad}{\text{Ad}}
\newcommand{\ad}{\text{ad}}
\newcommand{\op}{\text{opp}}
\newcommand{\tLg}{\tilde{\mathfrak{g}}}
\newcommand{\Lg}{\mathfrak{g}}
\newcommand{\Lb}{\mathfrak{b}}

\newcommand{\Gm}{\mathbb{G}_m}

\begin{document}
\title{A formula for certain Shalika germs of ramified unitary groups}
\date{\today}
\author{Cheng-Chiang Tsai}
\begin{abstract}
In this article, for nilpotent orbits of ramified quasi-split unitary groups with two Jordan blocks, we give closed formulas for their Shalika germs at certain equi-valued elements with half-integral depth previously studied by Hales \cite{Hales}. These elements are parametrized by hyperelliptic curves defined over the residue field, and the numbers we obtain can be expressed in terms of Frobenius eigenvalues on the $\ell$-adic $H^1$ of the curves, generalizing previous result of Hales on stable subregular Shalika germs. These Shalika germ formulas imply new results on stability and endoscopic transfer of nilpotent orbital integrals of ramified unitary groups. We mention also how the same numbers appear in the local character expansion of specific supercuspidal representations and consequently dimensions of degenerate Whittaker models.
\end{abstract}
\maketitle 

\tableofcontents

\section{Introduction}
We begin by introducing the unitary groups, related algebraic groups, Lie algebras and representations, and then the notion of Shalika germs. After that we can state our Shalika germ formulas, and describe its applications.\p

Let $F$ be a non-archimedean local field and $k$ its residue field. We fix an algebraic closure $\bar{k}$ of $k$. We assume $\text{char}(k)\not=2$. Let $E$ be a ramified quadratic extension over $F$. Note $E/F$ is tame. Fix in this article a uniformizer $\pi\in F$ whose square root $\pi^{1/2}\in E$. Let $n\ge 1$ be an integer and let $\tG=U_n(E/F)$ be the quasi-split unitary group of $n$ variables over $F$ which splits over $E$. We also assume either $\text{char}(F)=0$ or $\text{char}(F)>n$.\p

The reason for the notation $\tG$ is that we prefer to, just like Bruhat-Tits and in geometric Langlands, think of reductive groups over $F$ intuitively as an ind-pro-scheme over the residue field $k$. For this reason, in this article everything - groups, Lie algebras and their elements - that lives over $F$ will have its notation with a tilde $\widetilde{\;\;\;}$.\p

Fix a vertex $\bx$ on the Bruhat-Tits building of $\tG$ over $F$ whose reductive quotient is $\text{SO}_n(k)$. The vertex $\bx$ becomes hyperspecial after base change to $E$. The reductive quotient at $\bx$ over $E$ is (the $k$-points of) $G:=\text{GL}_n$. The root system of $G$ is in canonical bijection with the root system of $\tG/_E$, and we can choose compatible pinnings for $G$ and $\tG/_E$. The non-trivial element in $\text{Gal}(E/F)$ then provides an involution $\theta$ on $G$ such that the reductive quotient at $\bx$ over $F$ is $(G^{\theta})^o(k)\cong \text{SO}_n(k)$. A detailed and general construction of this is described in \cite[Sec. 4]{RY}.\p

Write $\tLg=\text{Lie }\tG$ and $\Lg=\text{Lie }G$. The involution $\theta$ also acts on $\Lg$. We'll write $G(0)=(G^{\theta})^o\cong\text{SO}_n/_k$, $\Lg(0)=\Lg^{\theta=1}$, and $\Lg(1)=\Lg^{\theta=-1}$. This provides a $\Z/2$-grading on $\Lg$. Write $V$ for the $n$-dimensional standard representation of $G(0)$ and $\Lg(0)$. We have $\Lg(1)\cong\Sym^2(V)$ as $G(0)$-representations. The Moy-Prasad filtration at $\bx$ jumps at half-integral numbers, and satisfies $\tG(F)_{\bx,0}/\tG(F)_{\bx,1/2}\cong G(0)(k)$, $\tLg(F)_{\bx,d/2}/\tLg(F)_{\bx,(d+1)/2}\cong\Lg(d)(k),\;\forall d\in\Z$, where the latter isomorphism is as a $G(0)(k)$-representation.\p 

Write $\Lg(1)^{rs}:=\Lg^{rs}\cap\Lg(1)$ where $\Lg^{rs}$ is the open subset of regular semisimple elements in the Lie algebra. Fix from now on a $T\in\Lg(1)^{rs}(k)$. We can see $T$ as a self-adjoint endomorphism on $V$. The monic characteristic polynomial $p_T$ is a separable polynomial of degree $n$. Consequently $C_T:=(y^2=p_T(x))$ is a hyperelliptic curve with genus $g=\floor{\frac{n-1}{2}}$. In fact, the representation $G(0)\curvearrowright\Lg(1)$ was first considered by Bhargava-Gross \cite{BG} for the study of arithmetic statistics about these hyperelliptic curves.\p

Consider the quotient map $\tLg(F)_{\bx,-1/2}\twoheadrightarrow\Lg(1)(k)$. Let $\til{T}\in\tLg(F)_{\bx,-1/2}$ be any lift. Such a $\til{T}$ is always regular semisimple and elliptic, i.e. $\Stab_{\tG}(\til{T})$ is an anisotropic torus over $F$. In fact, it's even anisotropic over $F^{ur}$. This implies that the orbits in the stable orbit of $\til{T}$ enjoys a bijection with the orbits in the stable orbit of $T$; see Lemma \ref{Galois}. (The notion of stable orbit is reviewed in Section \ref{term}.)\p

Denote by $\CO(0)$ the set of nilpotent orbits in $\tLg(F)$, and $J(\til{X},f)$ the orbital integral of $f$ on the orbit of $\til{X}\in\tLg$. We will often identify an element in $\tLg(F)$ with its orbit when talking about orbital integrals and Shalika germs. The theorem of Shalika \cite{Sh} asserts, for $\text{char}(F)=0$ or $\text{char}(F)\gg0$, the existence of constants, the {\bf Shalika germs} $\Gamma_{\CO}(\til{T})\in\Q$ such that

\mhp
\begin{equation}\label{Shalika}
J(\til{T},f)=\sum_{\CO\in\CO(0)}\Gamma_{\CO}(\til{T})J(\CO,f),
\end{equation}

for any compactly supported functions $f$ on $\tLg(F)$ that are locally constant by a sufficiently large lattice.\p

In this article, we prove the following theorem for Shalika germs of nilpotent elements $\til{N}_m\in\tLg(F)$, $0\le m\le g$ with two Jordan blocks of sizes $n-m$ and $m$. Denote by $q:=\#k$. Let $\lambda_1,\lambda_1',...,\lambda_g,\lambda_g'$ be Frobenius eigenvalues on $H^1(C_T/_{\bar{k}},\Ql)$, ordered so that $\lambda_i\lambda_i'=q$. Also write $\lambda_0=1$, $\lambda_0'=q$. Let $I=\{1,...,g\}$ if $n=2g+1$ and $I=\{0,1,...,g\}$ if $n=2g+2$. Write
\[a_m(T):=(-1)^m\cdot\!\!\sum_{S\subset I,|S|=m}\left(\prod_{i\in S}(\lambda_i+\lambda_i')\right).\]\hp

\begin{theorem}\label{intromain}(Theorem \ref{main} and \ref{even})
For $0\le m\le g$, we have $\Gamma_{\til{N}_m}^{st}(\til{T})=\pm a_m(T)$, where $\Gamma_{\CO}^{st}(\til{T})$ is the sum of $\Gamma_{\CO}(\til{T}')$ for $\til{T}'$ running over different orbits in the stable orbit of $\til{T}$.
\end{theorem}

See the theorems for the sign and see Appendix \ref{norm} for the normalization. When $m=0$, $\til{N}_0$ is a regular orbit and $a_0(T)=1$ which is well-known. When $m=1$ it's a subregular orbit, and the result was proven by Hales \cite{Hales}. He also gave parallel results for other classical groups. Our result probably brings the suggestion that general Shalika germs, after all, could have reasonably nice closed formulas.\p

The starting point of the proof of Theorem \ref{intromain} is to find a particular sequence of test functions for (\ref{Shalika}) supported on $\tLg(F)_{\bx,-1/2}$. These functions are made available by the homogeneity result of DeBacker (special case by Waldspurger) \cite{De}. The description of these functions will be given in the beginning of Section \ref{seccompute}. For these test functions, the LHS of (\ref{Shalika}) counts $k$-points on a sequence of specific varieties. It turns out that the theory of pencils of quadrics by X. Wang \cite{Jerry} can be used to relate these varieties to $\Sym^m(C_T)$, the $m$-th symmetric power of the hyperelliptic curve $C_T$. They then in terms give the numbers $a_m(T)$ above. This is the main concept in Section \ref{secgeom}. Briefly speaking, the phenomenon is that in Goresky-Kottwitz-MacPherson \cite{GKM} we know certain orbital integrals can be understood as counting points on Hessenberg varieties (their definition of Hessenberg varieties are more general than others', see \cite[1.5]{GKM}). The situation for obtaining Shalika germs could be slightly more involved and we end up with counting points on quasi-finite covers of Hessenberg varieties (see \cite{Ts}). Our quasi-finite covers of Hessenberg varieties then happens to be strongly related to varieties considered by Wang.\p

Section \ref{seccompute} contains most of the computation. We begin with the case of odd ramified unitary groups. In subsection \ref{secss} we read out the varieties that appear in the LHS of (\ref{Shalika}) for our test functions and apply the geometric result in Section \ref{secgeom}. Next in subsection \ref{secnil} we use Ranga Rao's method to compute nilpotent orbital integrals. With our simple-looking test functions thanks to homogeneity result of DeBacker, our computation reduces to a combinatorial sum over the Weyl group of $G(0)$. In subsection \ref{submore} we state results regarding Shalika germs (instead of stable Shalika germs), as well as the results for even quasi-split ramified unitary groups.\p

A consequence of Theorem \ref{intromain} is result regarding stable distributions supported on the nilpotent cone (i.e. linear combination of nilpotent orbital integrals) and endoscopic transfer of nilpotent orbital integrals of ramified quasi-split unitary groups. This is the main content of Section \ref{secendo}. The basic idea is that Shalika germs are the coefficients comparing regular semisimple orbital integrals and nilpotent orbital integrals. Once we know these coefficients, we are able to derive, from the very definition of stability and endoscopic transfer of regular semisimple orbital integrals, corresponding results of nilpotent ones.\p

The relevant elliptic endoscopic data are $U_{n_1}(E/F)\times U_{n_2}(E/F)$ with $n_1+n_2=n$ as endoscopy groups of $U_n(E/F)$. Assuming some conjectures of Assem (Conjecture \ref{Assem} and \ref{Assem2}), our result for nilpotent orbits with two Jordan blocks agrees with previous results of Waldspurger \cite{Wald} for unramified unitary groups. This also provides another evidence for Assem's conjectures. In fact, it was this connection to endoscopic transfer which led us into believing the formula in Theorem \ref{intromain} in the first place (see Remark \ref{rmkconj}).\p

In addition, in section \ref{charsc} we describe how those Shalika germs we compute show up in the Harish-Chandra-Howe local character expansions for some supercuspidal representations. Since {M\oe glin} and Waldspurger \cite{MW} showed that the coefficients in the loacl character expansions are related to the dimension of certain degenerate Whittaker models, we can produce examples where the dimension of degenerate Whittaker model are given by counting points on some ``non-elementary'' varieties.\p

%It's the very pleasure for the author to thank his advisor Benedict Gross for his suggestion on studying this problem and for his stimulating ideas and guidance. He will like to express his gratitude to X. Jerry Wang, for teaching him the theory of pencils of quadrics and its applications. Also the author would like to thank Sam Altschul for introducing him to the use of homogeneity result, Thomas Hales for many fundamental ideas on Shalika germs, and Jack Thorne for sharing valuable observation from arithmetic invariant theory. Lastly, he would like to thank Stephen DeBacker, Bao Le Hung, Fiona Murnaghan and Loren Spice for numerous inspiring discussions.

\hp
\subsection*{Acknowledgments}It's the very pleasure of the author to thank his advisor Benedict Gross for his suggestion on studying this problem and for his stimulating ideas and guidance. He will also like to express his gratitude to Xiaoheng Jerry Wang, for introducing him to the theory of pencils of quadrics and its applications. Meanwhile he would like to thank Zhiwei Yun, for teaching and sharing with him many brilliant ideas related to Hessenberg varieties. He has also learned a lot from Thomas Hales about many fundamental ideas on Shalika germs, for which he deeply appreciate. Lastly, he would like to express his gratitude to Sam Altschul, Stephen DeBacker, Jessica Fintzen, Bao Le Hung, Fiona Murnaghan, Loren Spice and Jack Thorne for numerous inspiring and helpful discussions.\p

\p
\section{Notations and setup}\label{term} We collect the notations. We have a non-archimedean local field $F$, its residue field $k$, a ramified quadratic extension $E/F$, and a fixed uniformizer $\pi\in F$ such that $\pi^{1/2}\in E$. Also we denote $q:=\#k$. We have $\til{G}=U_n(E/F)$ is a quasi-split unitary group that splits over $E$ (such group is unique). We write $\tLg=\text{Lie }\tG$. When orbital integral on $\tG$ or $\tLg$ is concerned, we always identify an element with its $\tG(F)$-orbit. The assumptions $\text{char}(k)\not=2$ and either $\text{char}(F)=0$ or $\text{char}(F)>n$ are imposed. In fact we'll mostly work with the assumption $\text{char}(k)\gg0$, and leave it to Appendix \ref{char} to explain how we can reduce the assumption on characteristic to those stated above. \p

Write $G=\text{GL}_n/_k$. It has a standard representation $V$, which we equipped with a non-degenerate quadratic form $\inn$. We define an involution $\theta$ on $G$ such that $\theta(h)=(h^t)^{-1}$ for $h\in G$, where $h^t$ is the transpose of $h$ with respect to $\inn$. This induces an involution on $\Lg:=\text{Lie }G$ which we also denote by $\theta$. Let $G(0)=(G^{\theta})^o\cong\text{SO}_n$, $\Lg(0)=\Lg^{\theta=1}=\text{Lie }G(0)$ and $\Lg(1)=\Lg^{\theta=-1}$ the invariant and anti-invariant subspace of $\theta$. There is a vertex $\mathbf{x}$ on the building such that $G$ is the reductive quotient of $\tG/_E$ at $x$ and $G(0)$ the reductive quotient of $\tG/_F$. We fix such a vertex $\mathbf{x}$. Also see subsection \ref{elem} below for a more elementary description of $\tG$, $G$ and $\bx$.\p

Let $\Lg^{rs}\subset\Lg$ be the subset of regular semisimple elements and $\Lg(1)^{rs}=\Lg^{rs}\cap\Lg(1)$. For any $T\in\Lg(1)^{rs}$, the monic characteristic polynomial of $T$ is denoted $p_T(x)$, and $C_T=(y^2=p_T(x))$ is the smooth completion of the hyperelliptic curve defined by $p_T(x)$.\p

Whenever we have a group variety $H$ acting on a space $X$ over some field $K$, by an {\it orbit} (or the orbit of $x\in X(K)$) in $X(K)$ we mean a subset of $X(K)$ of the form $\{h.x\,|\,h\in G(K)\}$, and by a stable orbit (or the stable orbit of $x$) we mean a subset of $X(K)$ of the form $\{h.x\,|\,h\in G(K^{sep})\}\cap X(K)$. The (stable) orbits discussed in this article will be either (stable) orbits in $\tLg(F)$ under the adjoint action of $\tG$, or (stable) orbits in $\Lg(1)(k)$ under the conjugacy action of $G(0)$.\p

The methods for odd ramified unitary groups ($n=2g+1$) and even (quasi-split) ramified unitary group ($n=2g+2$) are largely the same, but most of the computation has to be carried out separately. In most of this article we only treat the odd case in detail, but describe geometric tools needed for even unitary groups and list the results. In particular we will go with $\tG=U_{2g+1}(E/F)$ unless otherwise stated, and notationally reserve $n$ for other variables.\p

\subsection{An elementary description}\label{elem} We give a down-to-earth description of groups $\tG$, $G$, the involution $\theta$ and the vertex $\bx$. Let $\til{V}$ be an $n$-dimensional hermitian space over $E$, spanned by basis vectors $\til{e}_1$, ..., $\til{e}_n$ and equipped with the hermitian form given by $\pair{\sum a_i\til{e}_i}{\sum b_i\til{e}_i}_{herm}=\sum_{i=1}^n a_{n+1-i}b_i^*$, where $a_i,b_i\in E$ and $b_i^*$ is the conjugate of $b_i$ over $F$. Then $\tG$ is such an algebraic group defined over $F$ for which $\tG(F)$ is isomorphic to the group of unitary operators on $\til{V}$, i.e. $E$-linear operators on $\til{V}$ preserving the hermitian form.\p

Let $\Lambda=\text{span}_{\CO_E}\{\til{e}_1,...,\til{e}_n\}$ be a lattice in $\til{V}$. Let $K$ be the subgroup of $\tG(F)$ consisting of unitary operators $g$ with $g(\Lambda)=\Lambda$. Then $K$ stabilizes a unique vertex on the Bruhat-Tits building of $\tG$ over $F$, which (up to conjugation) is the vertex that we call $\bx$. We have the stabilizer group $\tG(F)_{\bx}=K$.\p

The hermitian form $\inn_{herm}$ takes $\CO_E$ values on $\Lambda$. Its reduction mod $\pi^{1/2}$ thus defines a quadratic form $\inn$ on $V:=\Lambda/\pi^{1/2}\Lambda$. Write $e_1,...,e_n$ to be the reduction of $\til{e}_1,...,\til{e}_n$, respectively. Then $\inn$ on $V$ is defined by $\pair{\sum a_ie_i}{\sum b_ie_i}=\sum_{i=1}^n a_{n+1-i}b_i$, where $a_i,b_i\in k$. The algebraic group $G$ then should be identified with the group of automorphisms of $V$ (not necessarily fixing $\inn$); $G(k')=GL(V\otimes_kk')$ for any finite extension $k'/k$, and $\theta\curvearrowright G(k')$ is the involution $\theta(g)=(g^t)^{-1}$ where $g^t$ denotes the transpose of $g$ with respect to the quadratic form $\inn$.\p

The Lie algebra $\tLg(F)$ is the space of anti-hermitian endomorphisms of $\til{V}$, and for any $d\in\frac{1}{2}\Z$, $\tLg(F)_{\bx,d}=\{X\in\tLg(F)\,|\,X(\Lambda)\subset\pi^d\Lambda\}$. We have $\tLg(F)_{\bx,d}/\tLg(F)_{\bx,d+1/2}\cong\Lg(2d)(k)$ given by first scaling $\pi^{-d}$ and then modulo $\pi^{1/2}$. Of course, this map depends on the choice of uniformizer $\pi^{1/2}\in E$. Here recall $\Lg(2d)=\Lg^{\theta=1}$ if $d$ is integral and $\Lg(2d)=\Lg^{\theta=-1}\cong Sym^2(V)$ if $d$ is half-integral but non-integral.\p

The algebraic group $G$ has $\theta$-stable Borel subgroups. For example, one such $B$ is given by that $B(k)$ consists of endomorphisms of $V$ that sends $e_i$ to a linear combination of $e_1$, $e_2$, ..., and $e_i$. We also denote $B(0)=B\cap G(0)=(B^{\theta})^o$. They are used in Section \ref{seccompute}.\p

\p
\section{Geometric result via pencils of quadrics}\label{secgeom}

In this section, $k$ can be any perfect field with $\text{char}(k)\not=2$.\p

\subsection{Odd case}\label{subodd} In this subsection we have $n=2g+1$ and $G=\text{GL}_{2g+1}/_k=\text{GL}(V)$. Recall that the vector space $V$ comes with a non-degenerate quadratic form $\inn$. We then have in the introduction an involution $\theta$ on $G$ which sends $g$ to $(g^t)^{-1}$, where $g^t$ is the adjoint of $g$ with respect to $\inn$. This induces an involution on $\Lg$, and we write $\Lg(0)=\Lg^{\theta=1}$, $\Lg(1)=\Lg^{\theta=-1}$. We have $\Lg(1)\cong\Sym^2(V)$ as $G(0)$-representations. As $\inn$ provides a self-dual structure on $V$, $\Lg(1)\cong\text{End}^{self-adj}(V)$ is also the space of self-adjoint operators on $V$.\p

The representation $G(0)\curvearrowright\Lg(1)$, or equivalently $\text{SO}(V)\curvearrowright\Sym^2(V)$, was considered by Bhargava-Gross in \cite{BG}. An orbit in this representation is GIT-stable iff it's contained in $\Lg(1)^{rs}:=\Lg^{rs}\cap\Lg(1)$ where $\Lg^{rs}$ is the open subset of regular semisimple elements in the Lie algebra. We now fix an $T\in\Lg(1)^{rs}(k)$.\p

Let $p_T(x)$ be the degree $2g+1$ monic characteristic polynomial of $T$.
Let $L=k[x]/p_T(x)$ be a degree $2g+1$ \'{e}tale algebra over $k$. Consider the Weil restriction $\text{Res}_k^L\mu_2$. This is a commutative \'{e}tale finite group scheme over $k$ of order $2^{2g+1}$. It has a surjective norm map $Nm:\text{Res}_k^L\mu_2\ra\mu_2$. Bhargava and Gross observed for $T\in\Lg(1)^{rs}$, we have canonical isomorphism $\Stab_{G(0)}(T)\cong\text{ker}(\text{Res}^{k[x]/p_T(x)}_k\mu_2\xra{Nm}\mu_2)$. In fact, the map $T\mapsto p_T(x)$ is the GIT-quotient map $\Lg(1)\mapsto\Lg(1)/\!/G(0)$; we have $\Lg(1)/\!/G(0)\cong\mathbb{A}^{2g+1}$ is the space of degree $n$ monic polynomials.\p

Let $C_T=(y^2=p_T(x))$ be a (smooth completion of) genus $g$ hyperelliptic curve. Let $J_T=\text{Pic}^0(C_T)$. Since the $2$-torsion $J_T[2]$ is generated by differences of Weierstrass points, one checks $J_T[2]\cong\text{ker}(\text{Res}^{k[x]/p_T(x)}_k\mu_2\xra{Nm}\mu_2)$. Consequently $J_T[2]\cong\Stab_{G(0)}(T)$.\p

If one fix such a $T$, then the orbit of $T$ is $G(0)(k).T$ while the stable orbit of $T$ is $\left(G(0)(\bar{k}).T\right)\cap\Lg(1)(\bar{k})$. There could be more than one orbits inside a stable orbit, and relative to the choice of $T$ as a pinning they can be classified by $H^1(k,\Stab_{G(0)}(T))\ra H^1(k,G(0))$. When $k$ is a finite field, by Lang's theorem, the latter pointed set is trivial, and thus we have $H^1(k,\Stab_{G(0)}(T))\cong H^1(k,J_T[2])$ classifies orbits in the stable orbit of $T$ relative to the choice of a pinning.\p

The GIT-quotient map $\Lg(1)\ra\Lg(1)/\!/G(0)$ has a {\it Kostant section} \cite[Thm 5.5]{Levy}. Using the Kostant section as a pinning, a $G(0)(k)$-orbit in $\Lg(1)^{rs}(k)$ corresponds to a hyperelliptic curve $C_T$ together with a class in $H^1(k,J_T[2])$. For $k$ a global field, Bhargava, Gross and others used this to study the average size of $2$-Selmer groups of such hyperelliptic curves, see e.g. \cite{BG2}. For this purpose, X. Wang developed the theory of pencil of quadrics \cite{Jerry}. It turns out that his theory is very useful in describing the variety that we'll encounter in orbital integrals.\p

Define on $V\oplus k$ two quadratic forms by $\langle (v_1,c_1),(v_2,c_2)\rangle_1=\langle v_1,v_2\rangle$ and $\langle (v_1,c_1),(v_2,c_2)\rangle_2=\langle v_1,Tv_2\rangle-c_1c_2$. This defines a generic pencil of quadrics in the sense of X. Wang \cite[Intro.]{Jerry}. Recall that a subspace $W\subset V\oplus k$ is said to be isotropic with respect to a quadric (e.g. $\inn_1$) if the restriction of the quadratic form to $W$ is trivial. In his paper, Wang proved the following:\p

\begin{theorem}\label{Jerry}(Wang \cite[Thm. 2.26]{Jerry}) Let $F_T$ be the variety that parametrizes $g$-dimensional subspaces of $V\oplus k$ that the are isotropic with respect to both $\inn_1$ and $\inn_2$. Then there is a commutative algebraic group structure on

\mhp
$$G_T:=J_T\sqcup F_T\sqcup\text{Pic}^1(C_T)\sqcup F_T',$$\hp

where $F_T'\cong F_T$ as a variety, the addition law on $J_T\sqcup\text{Pic}^1(C_T)$ agrees with that of $\text{Pic}(C_T)/(2(\infty)=0)$, and $G_T$ has component group equal to $\Z/4$.
\end{theorem}\hp

In particular, $F_T$ is a torsor under $J_T$ and there is a doubling map $\times2:F_T\ra\text{Pic}^1(C_T)$. We review the group structure in the theorem. The group structure is determined by $(p)-[W]$, i.e. how to subtract from $p\in C_T$ a subspace $[W]\in F_T$. This is done as follows: a point $p=(x,y)$ on $C_T$ corresponds to a ruling of $\inn_2-x\inn_1$. Recall that a ruling is a connected component of the variety parameterizing $(g+1)$-dimensional subspace on which the quadratic form is trivial. There will be a unique $(g+1)$-dimensional space $W'$ in the ruling such that $W'\supset W$. Inside the space $W'$ there will be, when counted with multiplicity, two $g$-dimensional subspaces $W$ and $W''$ on which $\inn_1$ and $\inn_2$ vanish. It is then defined $(p)-[W]:=[W'']$, and this uniquely characterizes the group structure on $G_T$.\p 

It's obvious that $G_T$ depends only on $T$ up to $G(0)(k)$-conjugacy. As mentioned in the introduction the orbit of $T$ in its stable orbit may be characterized by a class in $H^1(k,J_T[2])$. This class can be describe as follows: the map $\times2:F_T\ra\text{Pic}^1(C_T)$ is \'{e}tale and Galois with Galois group being $J_T[2]$ as a group scheme over $k$.  There is a distinguished rational Weierstrass point $\infty\in C_T\subset\text{Pic}^1(C_T)$. Then the class is the torsor $(\times2)^{-1}(\infty)$.\p

For any $0\le m\le g$, consider $j_m:\Sym^m(C_T)\ra\text{Pic}^1(C_T)$ by $j_m(p_1,...,p_m)=(p_1)+...+(p_m)-(m-1)(\infty)$. Let $X_{T,m}$ be the image of $j_m$, and let $\til{X}_{T,m}:=(\times 2)^{-1}(X_{T,m})$ be its preimage under the \'{e}tale map $\times2$. We also take $\til{X}_{T,-1}=\emptyset$. We shall relate $\til{X}_{T,m}$ with the following varieties $F_{T,m}$, which could be thought as a generalized version of Hessenberg varieties considered by Goresky, Kottwitz and MacPherson \cite{GKM}.\p

For any finite extension $k'/k$, we call a flag of $k'$-subspaces $0\subset W^1\subset...\subset W^g\subset(V\oplus k)\otimes_k k'$ {\bf good} if\hp

(i)$\;\:\,$ $\dim W^i=i$.

(ii)$\:\:$ The restriction of $\inn_1$ and $\inn_2$ to $W^g$ is zero.

(iii)$\:$ $W^{g-1}\subset V\otimes_k k'$.

(iv)$\,\,$ $T(W^i)\subset W^{i+2}\;\forall 1\le i\le g-3$.

(v)$\;\;$ $T(W^{g-2})\subset\pi_1(W^g)$.\hp

Here $\pi_1:(V\oplus k)\otimes_k k'\ra V\otimes_k k'$ is the projection to the first factor. For $1\le m\le g$, a good flag is called {\bf $m$-good} if $T(W^{g-m})\subset W^{g-m+1}$ (where $W^0=\{0\}$ is understood). Also a good flag is called {\bf $0$-good} if $W^g\subset V\otimes_kk'$. Next, for $0\le m\le g$, an $m$-good flag is called {\bf $m$-excellent} if it is also $n$-good for $m<n\le g$. On the other hand, a good flag is called {\bf $m$-general} if it is not $n$-good for any $0\le n<m$. Finally, a good flag is called {\bf $m$-exact} if it is $m$-excellent and $m$-general. Now let
$$F_{T,m}(k')=\{0\subset W^1\subset...\subset W^g\subset(V\oplus k)\otimes_k k'\,|\,\text{This is an }m\text{-exact flag}\}.$$\p

The functor $F_{T,m}$ is easily seen from its very definition to be represented by a quasi-projective variety over $k$ which we'll denote with the same notation. In fact, there is a projective variety $F_{T,good}$ that parameterize good flags, and $F_{T,m}\subset F_{T,good}$ is locally closed. There is a natural map $\til{j}:F_{T,good}\ra F_T$ by sending a flag to $[W^g]$. This section is devoted to the proof of the following result:\hp

\begin{theorem}\label{geom}
For $0\le m\le g$, the map $\til{j}|_{F_{T,m}}:F_{T,m}\ra F_T$ is a locally closed embedding, with image equal to $\til{X}_{T,m}\,\bsl\,\til{X}_{T,m-1}$. 
\end{theorem}\qp

A more direct proof of this theorem in the case $m\le 2$ was shown to me by X. Wang. Nevertheless, we begin our proof for the general case with two simple lemmas:\p

\begin{lemma}\label{rs}
If $0\subsetneq W\subsetneq V$ is such that $\inn$ is trivial on $W$, then $T(W)\not=W$.
\end{lemma}

\begin{proof}Let $W^{\perp}=\{v\in V\,|\,\langle v,w\rangle=0,\;\forall w\in W\}$. Then $W\subset W^{\perp}$ by assumption. Suppose on the contrary $T(W)=W$, then by adjointness $T(W^{\perp})=W^{\perp}$, and $T|_W$ is the adjoint of $T|_{V/W^{\perp}}$. But this says $T|_W$ and $T|_{V/W^{\perp}}$ have the same eigenvalues. Hence $T$ cannot be regular semisimple.
\end{proof}\qp

In the rest of this section, we work ``geometrically,'' {\it i.e. we replace $k$ by an algebraic closure $\bar{k}$}, so that we can omit the notations $\cdot\otimes_kk'$ and so on. This will make no harm to what we want to prove.\p

\begin{lemma}\label{uniqueflag} Let $0\le m\le g$.\qp

(i)$\;\;$ Suppose $0\subset W^1\subset...\subset W^g\subset V\oplus k$ and $0\subset (W^1)'\subset...\subset (W^g)'\subset V\oplus k$ are such that $W^g=(W^g)'$. If one of the flags is $m$-general, then $W^i=(W^i)'$ for $g-m\le i\le g$. In particular the other is also $m$-general.\qp

(ii)$\;\,$ If $0\subset W^1\subset...\subset W^g\subset V\oplus k$ is $m$-good, then there is a unique $0\subset (W^1)'\subset...\subset (W^g)'\subset V\oplus k$ which is $m$-excellent such that $W^i=(W^i)'$ for $g-m\le i\le g$.\qp

(iii)$\;$ If in (i) both flags are $m$-exact, then the two flags are the same.\qp
\end{lemma}

\begin{proof}
If $m\ge 1$, then $W^{g-1}=W^g\cap V$ is unique. Next if $m\ge 2$, then the flag is not $1$-good and $T(W^{g-1})\not\subset W^g$. Since goodness requires $T(W^{g-2})\subset W^g$, we have $W^{g-2}=W^{g-1}\,\cap\, T^{-1}(W^g)$ is also unique. Proceed similarly and we have the uniqueness of $W^{g-1}$, ..., $W^{g-m}$. This proves (i). Now suppose the flag is $m$-good and $(m+1)$-good. Then $T(W^{g-m})\subset W^{g-m+1}$ and $T(W^{g-m-1})\subset W^{g-m}$. However the previous lemma implies $T(W^{g-m})\not\subset W^{g-m}$. Thus $W^{g-m-1}=W^{g-m}\,\cap\, T^{-1}(W^{g-m})$ is the only possibility for this to hold, i.e. for the flag to be $(m+1)$-good. Continue the argument and we obtain (ii), and (iii) follows immediately.
\end{proof}\hp

Let $F^{ex}_{T,m}\subset F_{T,good}$ be projective varieties parameterizing $m$-excellent flags. The key is\p

\begin{lemma}\label{key} $\til{j}(F^{ex}_{T,m})\subset\til{X}_{T,m}$. Also when $m=0$, $\til{j}(F^{ex}_{T,0})=\til{X}_{T,0}$.
\end{lemma}\hp

\begin{lemma} Lemma \ref{key} above implies Theorem \ref{geom}.
\end{lemma}

\begin{proof} The second statement in Lemma \ref{key} gives the theorem when $m=0$. We now use induction on $m$. Let $\hat{F}_{T,m}\subset F^{ex}_{T,m}$ be the open subvariety that parameterize those flags that are $m$-excellent and $(m-1)$-general. We have $F_{T,m-1}\subset\hat{F}_{T,m}$ as a closed subvariety. By induction $\til{j}$ gives an isomorphism $F_{T,m-1}\cong\til{X}_{T,m-1}\,\bsl\,\til{X}_{T,m-2}$, which is $(m-1)$-dimensional. In particular $F_{T,m-1}\subset\hat{F}_{T,m}$ are both non-empty.\p

On the other hand, a dimension count shows that $F_{T,good}$ has dimension at least $g$, and $F^{ex}_{T,m}\subset F_{T,good}$ is a closed subvariety cut out by $g-m$ equations. As $\hat{F}_{T,m}$ is open in $F^{ex}_{T,m}$, every component of $\hat{F}_{T,m}$ has dimension at least $m$. This says that $F_{T,m}=\hat{F}_{T,m}\,\bsl\,F_{T,m-1}$ is non-empty (as a variety). Lemma \ref{key} will force the image of $F_{T,m}$ under $\til{j}$ to be inside the $m$-dimensional locus $\til{X}_{T,m}$, and Lemma \ref{uniqueflag}(iii) says that the dimension of the image has to be the same as the domain. Since $\dim \til{X}_{T,m}=m$, we have $\til{j}(F_{T,m})\subset\til{X}_{T,m}$ is dense.\p

Since $F^{ex}_{T,m}$ is proper, $\til{j}(F^{ex}_{T,m})=\til{X}_{T,m}$. We also have $\til{j}(F^{ex}_{T,m-1})=\til{X}_{T,m-1}$ by induction. By Lemma \ref{uniqueflag}(ii), the image of $F^{ex}_{T,m}\,\bsl\,F_{T,m}$ under $\til{j}$ is in $\til{X}_{T,m-1}$. By Lemma \ref{uniqueflag}(i), the image of $F_{T,m}$ is disjoint from $\til{X}_{T,m-1}$. Thus $\til{j}(F_{T,m})=\til{X}_{T,m}\,\bsl\,\til{X}_{T,m-1}$. The proof of the uniqueness in Lemma \ref{uniqueflag}(i) can be carefully checked to imply that not only $\til{j}$ is injective on closed point, but also $\til{j}:F_{T,m}\cong\til{X}_{T,m}\,\bsl\,\til{X}_{T,m-1}$ is an isomorphism.
\end{proof}\hp

\begin{proof}[Proof of Lemma \ref{key}]The case $m=0$, including the second statement, is precisely \cite[Remark 2.30]{Jerry}. We review it here.\p

The variety $\til{X}_{T,0}$ is by definition $(\times2)^{-1}(\infty)$. Let $\tau(\infty):V\oplus k\ra V\oplus k$ be the map sending $(v,c)$ to $(v,-c)$. Then for any $[W]\in F_T$, one checks from the definition that $[\tau(W)]=(\infty)-[W]$. The variety $(\times2)^{-1}(\infty)$ thus parameterizes $g$-dimensional varieties $W^g$ in $V\oplus k$, stabilized by $\tau(\infty)$, on which $\inn_1$ and $\inn_2$ vanish.\p

Since $\inn_2$ does not vanish on the second factor of $V\oplus k$, this forces $W^g$ to lie completely in the first factor, i.e. $(\times2)^{-1}(\infty)$ parameterizes $g$-dimensional varieties $W^g$ in $V$ on which $\inn_1|_V=\inn$ and $\inn_2|_V=\langle\cdot,T\cdot\rangle$ vanish. Lemma \ref{uniqueflag}(ii),(iii) then says it extends uniquely to a $0$-exact flag (there we began with a good flag that is $0$-good instead of only $W^g$, but the same proof also applies to the present case). This proves the $m=0$ statement.\p

The $J_T[2]$-structure on $(\times2)^{-1}(\infty)$ can be described as follows: Let $p_0,...,p_{2g},\infty$ be the Weierstrass points of $C_T$. Then $J_T[2]$ is generated by $((p_i)-(\infty))$, $0\le i\le 2g$ with the only relation that $\sum_{i=0}^{2g}((p_i)-(\infty))=0$. The action of $(p_i)-(\infty)$ is described as follows. Say $p_i=(x,0)$. Then $x\inn_1-\inn_2$ is a degenerate quadric with $1$-dimensional kernel $U$. For $[W]\in(\times2)^{-1}(\infty)$, $x\inn_1-\inn_2$ is trivial on the $(g+1)$-dimensional space $W+U$. There will be exactly one $g$-dimensional subspace $W'\subset W+U$, other than $W$, on which $\inn$ and $\langle\cdot,T\cdot\rangle$ are also trivial. It is then defined $((p_i)-(\infty)).[W]=[W']$.\p

We now discuss the general case. From now on $0<m\le g$ is fixed. We shall show $\til{j}(F^{good}_{T,m})\subset\til{X}_{T,m}$ for a {\it generic} $T\in\Lg(1)^{rs}$ (i.e. for $T$ in a Zariski open subset of $\Lg(1)^{rs}$). In fact, what we will do is the following: Fix a flag $\mathbb{F}=(0\subset W^1\subset...\subset W^{g-1}\subset W^g\subset V\oplus k)$ such that $W^{g-1}\subset V$, $W^g\not\subset V$, $k\not\subset W^g$ (here $V$ and $k$ are the first and the second components in $V\oplus k$), and $\inn$ is trivial on $\pi_1(W^g)$. Since all such flags in $V\oplus k$ are conjugate by $G(0)$ (where $G(0)$ preserves $V\subset V\oplus k$ and acts trivially on the second component), without loss of generality we may assume that $\mathbb{F}$ is exactly the flag in interest.\p

There is an irreducible closed subvariety $\mathcal{V}\subset\Lg(1)$ such that the flag is $m$-excellent with respect to $T\in\Lg(1)^{rs}$ if and only if $T$ lies inside $\mathcal{V}$. There is an Zariski open subset of $\mathcal{V}$ consisting of those $T$ for which the flag is $m$-exact. What we shall prove is that for an even smaller open subset $\Delta\subset\mathcal{V}$, all $T\in\Delta$ satisfy $\til{j}(\mathbb{F})\in\til{X}_{T,m}$. A continuity argument by having Theorem \ref{Jerry} in family then extends the result to all $T\in\mathcal{V}$, which is what we need.\p

The case $m=g$ is trivial, and we'll assume $0<m<g$. Let $0\subset W^1\subset...\subset W^g\subset V\oplus k$ be an $m$-good flag. By assumption $T(W^{g-m})\subset W^{g-m+1}$. Consider $U_0=W^{g-m}$ and $U^0=(W^{g-m+1})^{\perp1}:=\{(v,c)\in V\oplus k\,|\,\langle v,w\rangle=0,\forall w\in W^{g-m+1}\}$ (that is, ${}^{\perp1}$ is used to denote the orthogonal complement with respect to $\inn_1$). Following the spirit of \cite[Sec. 3.1]{Jerry}, we define inductively subspaces $U_0\subset U_1\subset...\subset U_{\lfloor\frac{m+1}{2}\rfloor}\subset U^{\lfloor\frac{m+1}{2}\rfloor}\subset U^{\lfloor\frac{m+1}{2}\rfloor-1}\subset...\subset U^0$ as follows:
\[\def\arraystretch{0.6}
\left\{\begin{matrix}
U_n&:=&(U^{n-1})^{\perp1}\cap W^{g-m+2n-1}.\\ \\
U^n&:=&(U_n)^{\perp2}\cap U^{n-1}.
\end{matrix}\right.
\]\qp

\begin{lemma}\label{induct}For $0\le n<\lfloor\frac{m+1}{2}\rfloor$, we have\qp

(i)$\;$ $U^n\supset (W^{g-m+2n+1})^{\perp1}$.\qp

For $0\le n\le\lfloor\frac{m+1}{2}\rfloor$ we have\qp

(ii)$\;\;\;$ $\dim U_n=g-m+n$.

(iii)$\;\;$ $\dim U^n=g+m-n+1$.
\end{lemma}

\begin{proof} For $n=0$ it's obvious. We now do induction on $n$. 
For (i) let $0<n<\lfloor\frac{m+1}{2}\rfloor$. Since $U_n\subset W^{g-m+2n-1}\subset W^{g-1}\subset V$, we have $U^n\supset k$. Now $\mathbb{F}$ is good (with respect to $T$) says $T(W^{g-m+2n-1})\subset\pi_1(W^{g-m+2n+1})$. Since $U_n\subset W^{g-m+2n-1}$. By definition of $\inn_2$ we have $(U_n)^{\perp2}\supset (W^{g-m+2n+1})^{\perp1}$. Also by induction $U^{n-1}\supset(W^{g-m+2n-1})^{\perp1}\supset(W^{g-m+2n+1})^{\perp1}$. This gives (i).\p

For (ii), since $U^{n-1}\supset(W^{g-m+2n-1})^{\perp1}$ by (i), we have $k\subset (U^{n-1})^{\perp1}\subset W^{g-m+2n-1}+k$, where $k$ denotes the second component in $V\oplus k$, i.e. the kernel of $\inn_1$. Since $k\not\subset W^{g-m+2n-1}$, this says $\dim U_n=\dim (U^{n-1})^{\perp1}-1=g-m+n$. This proves (ii).\p

Lastly for (iii), by definition $U^{n-1}=(U_{n-1})^{\perp2}\cap U^{n-1}$ (It's important for our construction that it holds when $n=1$!). Since $U_n$ contains $U_{n-1}$ with codimension $1$ by (ii), it suffices to show $(U_n)^{\perp2}\not\supset U^{n-1}$. Suppose on the contrary $(U_n)^{\perp2}\supset U^{n-1}$, then we have $U^{n-1}\subset (U_n)^{\perp1}\cap (U_n)^{\perp2}$. In this case we have $k\subset U^{n-1}\Rightarrow U_n\subset V=k^{\perp2}$. From the definition of $\inn_1$ and $\inn_2$ we get $\pi_1(U^{n-1})\subset (U_n)^{\perp}\cap T(U_n)^{\perp}$, where this time ${}^{\perp}$ means the orthogonal complement in $V$ with respect to $\inn$. But this is impossible, because by Lemma \ref{rs} $(U_n)^{\perp}\cap T(U_n)^{\perp}$ intersect non-trivially and thus have dimension less than $g-m+n$ by (ii), while $\dim\pi_1(U^{n-1})=\dim U^{n-1}-1=g+m-n+1$ by inductive hypothesis from (iii).
\end{proof}\hp

We now come back to the proof of Lemma \ref{key}. Define $L$ to be the variety that parametrize $g$-dimensional subspaces $W$ satisfying $U_{\lfloor\frac{m+1}{2}\rfloor}\subset W\subset U^{\lfloor\frac{m+1}{2}\rfloor}$ and that $\inn_1$ and $\inn_2$ vanish on $W$. In particular by construction we have $U_{\lfloor\frac{m+1}{2}\rfloor}\subset W^{g-m+2\lfloor\frac{m+1}{2}\rfloor-1}\subset W^g$ and $W^g\subset (W^g)^{\perp2}\cap U^0\subset U^{\lfloor\frac{m+1}{2}\rfloor}$, i.e. $[W^g]\in L$.\p

Define $\overline{V}$ to be the subquotient $\overline{V}:=U^{\lfloor\frac{m+1}{2}\rfloor}/U_{\lfloor\frac{m+1}{2}\rfloor}$. Since $U_{\lfloor\frac{m+1}{2}\rfloor}\subset(U^{\lfloor\frac{m+1}{2}\rfloor-1})^{\perp1}\subset(U^{\lfloor\frac{m+1}{2}\rfloor})^{\perp1}$ and $U^{\lfloor\frac{m+1}{2}\rfloor}\subset(U_{\lfloor\frac{m+1}{2}\rfloor})^{\perp2}\Rightarrow U_{\lfloor\frac{m+1}{2}\rfloor}\subset(U^{\lfloor\frac{m+1}{2}\rfloor})^{\perp2}$, the two quadratic forms $\inn_1$ and $\inn_2$ restricts to be quadratic forms on $\overline{V}$. Denote still their restrictions by $\inn_1$ and $\inn_2$, respecitvely. Let $L'$ be the variety that parameterize ${\lfloor\frac{m}{2}\rfloor}$-dimensional subspaces in $\overline{V}$ on which $\inn_1$ and $\inn_2$ are trivial. Then evidently $L\cong L'$.\p

Consider the polynomials $p_T^{(i)}(x)=\text{disc}\left(\inn_1-x\inn_2\right)|_{U^i/U_i}$ for  $i=0,1,...,\lfloor\frac{m+1}{2}\rfloor$. We claim $p_T^{(0)}(x)=x^{2\lfloor\frac{m+1}{2}\rfloor}p_T^{(\lfloor\frac{m+1}{2}\rfloor)}(x)$. To prove this, observe that when we go from $U^0/U_0$ to $U^1/U_1$, we quotient out $U_1/U_0$, on which both $\inn_1$ and $\inn_2$ is zero. Even more, $U_1/U_0$ is in the kernel of $\inn_1$ while $\pair{U_1/U_0}{U^0/U^1}_2$ is non-trivial. This exactly says $p_T^{(0)}(x)=x^2p_T^{(1)}(x)$. Repeating the argument gives the asserted result.\p

$L'$ comes from the situation in \cite[Sec. 2.1]{Jerry}, namely that of a generic pencil of odd dimension. As explained in the beginning of this proof, \cite[Remark 2.30]{Jerry} says that $L'$ is a torsor under the $2$-torsions of the Jacobian of a hyperelliptic curve of genus $\lfloor\frac{m}{2}\rfloor$. And this hyperelliptic curve is simply $\bar{C}_T=(y^2=p_T^{(\lfloor\frac{m+1}{2}\rfloor)}(x))$.\p

Recall $T\in\mathcal{V}$ is such that our flag $\mathbb{F}$ is $m$-excellent with respect to $T$. One checks from definition that when $T$ runs over the subspace $\mathcal{V}$, $p_T^{(0)}(x)$ runs over all polynomials of degree less than or equal to $2m+1$ that are divided by $x^{m+1}$. Consequently, there exists a Zariski open subset $\Delta\subset\mathcal{V}$ such that for $T\in\Delta$, $p_T^{(\lfloor\frac{m+1}{2}\rfloor)}(x)$ is separable (i.e. having distinct roots).\p

For such $T$, let $p_0,...,p_{2\lfloor\frac{m}{2}\rfloor},\infty$ be the Weierstrass points of this hyperelliptic curve. We have $((p_0)-(\infty))+...+((p_{2\lfloor\frac{m}{2}\rfloor})-(\infty))=0\in\text{Pic}^0(\bar{C}_T)$. By the $\text{Pic}^0(\bar{C}_T)[2]$-torsor structure on $L'$ as described in the $m=0$ case at the beginning of this proof, we have a sequence $\Omega_0$, $\Omega_1$, ..., $\Omega_{2\lfloor\frac{m}{2}\rfloor}$, $\Omega_{2\lfloor\frac{m}{2}\rfloor+1}=\Omega_0$ such that $\Omega_i$ and $\Omega_{i+1}$ intersect in codimension $1$. When $m$ is even, $\bar{p}_T(x)$ is divisible by $x$, which means that one of the Weierstrass point $p_i=(0,0)$, i.e. the quadric $\inn_1$ is degenerate on $\overline{V}$, and $\inn_1$ vanish on the sum $\Omega_i+\Omega_{i+1}$ for some $i$.\p

Now recall that $\overline{V}$ is a subquotient of $V\oplus k$. The preimage of $\Omega_0$, $\Omega_1$, ..., $\Omega_{2\lfloor\frac{m}{2}\rfloor}$, $\Omega_{2\lfloor\frac{m}{2}\rfloor+1}=\Omega_0$ is a sequence $\hat{\Omega}_0$, ..., $\hat{\Omega}_{2\lfloor\frac{m}{2}\rfloor+1}$ of subspaces in $V\oplus k$ with $\hat{\Omega}_0=\hat{\Omega}_{2\lfloor\frac{m}{2}\rfloor+1}=W^g$, and that $\hat{\Omega}_i$ and $\hat{\Omega}_{i+1}$ intersects in codimension $1$. This says that there are points $p_1,....,p_{2\lfloor\frac{m}{2}\rfloor+1}$ on $C_T$ such that $((p_1)-((p_2)-...-((p_{2\lfloor\frac{m}{2}\rfloor+1})-[W^g])...))=[W^g]$, and one of $p_i$ is $\infty$ when $m$ is even. By definition of the group structure, this says $(\times2)([W^g])\in X_{T,m}$. We have therefore finished the proof of the lemma and hence that of Theorem \ref{geom}.
\end{proof}\hp

\begin{remark} I first learned from Jack Thorne the idea that symmetric powers of $C_T$ should arise, which he observed in his unpublished work generalizing his results in \cite{Th} to nilpotent orbits of two Jordan blocks in type $\mathbf{A}$. In fact, if one substitutes this whole section with corresponding result in \cite[Thm 3.7]{Th} and re-apply the method in Section \ref{seccompute}, one can obtain stable Shalika germs for subregular nilpotent orbits in terms of number of rational points on the curves in \cite[Thm 3.7]{Th}. This shows, for example, that certain stable subregular Shalika germs of $E_6$ (resp. $E_7$, $E_8$) will be given by counting points on non-hyperelliptic curves of genus $3$ (resp. $3$, $4$).
\end{remark}\hp

\hp
\subsection{Even case}\label{subeven} In this subsection $n=2g+2$; $\tG=U_{2g+2}(E/F)$, $G=\text{GL}_{2g+2}/_k=\text{GL}(V)$, $G(0)=\text{SO}(V)$ and $\Lg(1)=\Sym^2(V)$ where $V$ is a $2g+2$ dimensional non-degenerate split quadratic space. The method in this subsection is almost identical to that of the previous one, and we only list the setting, definition and results here.\p

We have parallel result to Theorem \ref{geom}. Again fix $T\in\Lg(1)^{rs}(k)$. We also write $p_T(x)\in k[x]$ the monic characteristic polynomial of $T$. Our hyperelliptic curve $C_T:=(y^2=p_T(x))$ now has two points above the infinity on $\mathbb{P}^1$. We shall denote these two points by $\infty^{(1)}$ and $\infty^{(2)}$. They are both defined over $k$.\p

Consider $L=k[x]/p_T(x)$. The Weil restriction $\text{Res}_k^L\mu_2$ now has not only a surjective norm map $Nm:\text{Res}_k^L\mu_2\ra\mu_2$ but also a diagonal embedding $\Delta:\mu_2\ra\text{Res}_k^L\mu_2$. We have $\Stab_{\text{O}(V)}(T)\cong\text{Res}_k^L\mu_2$, also $\Stab_{G(0)}(T)\cong\ker(\text{Res}_k^L\mu_2\xra{Nm}\mu_2)$, and lastly $J_T[2]\cong\Stab_{G(0)}(T)/Z(G(0))\cong\left(\ker(\text{Res}_k^L\mu_2\xra{Nm}\mu_2)\right)/\Delta(\mu_2)$.\p

We denote by $\inn_1=\inn$ the standard quadratic form on $V$, i.e. the one which is invariant by $G(0)$. Then $\infty^{(1)}$ and $\infty^{(2)}$ are just the two rulings of $\inn_1$. Define $\inn_2$ on $V$ by $\pair{v_1}{v_2}_2=\pair{v_1}{Tv_2}_1$. Then the theory of pencil of quadrics says:\hp

\begin{theorem}\label{Jerry2} (Wang \cite[Thm. 2.26]{Jerry}) Let $F_T$ be the variety that parameterizes $g$-dimensional subspaces of $V$ that are isotropic with respect to $\inn_1$ and $\inn_2$. Then there is a commutative algebraic group structure on

\mhp
$$G_T:=J_T\sqcup F_T\sqcup\text{Pic}^1(C_T)\sqcup F_T',$$\hp

where $F_T'\cong F_T$ as a variety, the addition law on $J_T\sqcup\text{Pic}^1(C_T)$ agrees with that of $\text{Pic}(C_T)/((\infty^{(1)})+(\infty^{(2)})=0)$, and $G_T$ has component group equal to $\Z/4$.
\end{theorem}\hp

We again write the doubling map $\times2:F_T\ra\text{Pic}^1(C_T)$ which is \'{e}tale Galois with Galois group $J_T[2]$. For $0\le m\le g$ with $m$ even, define $j_m^{(1)},j_m^{(2)}:\text{Sym}^m(C_T)\ra\text{Pic}^1(C_T)$ by \[j_m^{(1)}(p_1,...,p_m)=(p_1)+...+(p_m)-\left(\frac{m}{2}-1\right)(\infty^{(1)})-\frac{m}{2}(\infty^{(2)}).\] \[j_m^{(2)}(p_1,...,p_m)=(p_1)+...+(p_m)-\frac{m}{2}(\infty^{(1)})-\left(\frac{m}{2}-1\right)(\infty^{(2)}).\]\hp

And we define $X_{T,m}^{(i)}$ to be the image of $j_m^{(i)}$, and $\til{X}_{T,m}^{(i)}=(\times2)^{-1}\left(X_{T,m}^{(i)}\right)$, $i=1,2$.\p

For $0<m\le g$ with $m$ odd, we define $j_m^{(0)},j_m^{(1)},j_m^{(2)}:\text{Sym}^m(C_T)\ra\text{Pic}^1(C_T)$ by
\[j_m^{(0)}(p_1,...,p_m)=(p_1)+...+(p_m)-\frac{m-1}{2}(\infty^{(1)})-\frac{m-1}{2}(\infty^{(2)}).\]
\[j_m^{(1)}(p_1,...,p_m)=(p_1)+...+(p_m)-\frac{m-3}{2}(\infty^{(1)})-\frac{m+1}{2}(\infty^{(2)}).\]
\[j_m^{(2)}(p_1,...,p_m)=(p_1)+...+(p_m)-\frac{m+1}{2}(\infty^{(1)})-\frac{m-3}{2}(\infty^{(2)}).\]\hp

And we define $X_{T,m}^{(i)}$ to be the image of $j_m^{(i)}$, and $\til{X}_{T,m}^{(i)}=(\times2)^{-1}\left(X_{T,m}^{(i)}\right)$, $i=0,1,2$.\p

Next, we introduce the notion of good flags. A flag of subspaces $0\subset W^1\subset...\subset W^{g+1}\subset V$ is called {\bf good} if\hp

(i)$\;\;\;$ $\dim W^i=i$.

(ii)$\;\;$ The restriction of $\inn_1$ to $W^{g+1}$ is zero.

(iii)$\;$ The restriction of $\inn_2$ to $W^g$ is zero.

(iv)$\,\,$ $T(W^i)\subset W^{i+2}$, $\forall 1\le i\le g-1$.\p

For $0\le m\le g$, a good flag is called {\bf $m$-good} if $T(W^{g-m})\subset W^{g-m+1}$. Here $W^{-1}=W^0=\{0\}$, i.e. good flags are automatically $g$-good. A flag is called {\bf $m$-excellent} if it is $n$-good for $m\le n\le g$.\p

For $0\le m\le g$, a good flag is called {\bf $m$-general} if it is not $n$-good for any $0\le n<m$. For any $0<m\le g$, we now define the notion of {\bf $m$-exact} flags\footnote{See also Remark \ref{halfexact}.}. Let $\{W^r\}_{r=1}^{g+1}$ be any $m$-excellent and $m$-general flag. There always exists another $m$-excellent flag $\{U^r\}_{r=1}^{g+1}$ satisfying $U^g=W^g$ but $U^{g+1}\not=W^{g+1}$ if $m$ is odd, or $U^{g-1}=W^{g-1}$, $U^{g+1}=W^{g+1}$ but $U^g\not=W^g$ if $m$ is even. We say $\{W^r\}_{r=1}^{g+1}$ is $m$-exact if $\{U^r\}_{r=1}^{g+1}$ is also $m$-general. Lastly, a $0$-excellent flag is said to be $0$-exact.\p

Now let $F_{T,m}^{(1)}$ be the variety that parameterize $m$-exact flags for which $W^{g+1}$ is in the ruling $\infty^{(1)}$, and $F_{T,m}^{(2)}$ be the variety that parameterizes those $m$-exact flags for which $W^{g+1}$ is in the other ruling $\infty^{(2)}$. We have natural maps $\til{j}:F_{T,m}:=F_{T,m}^{(1)}\sqcup F_{T,m}^{(2)}\ra F_T$ by sending $\{W^r\}_{r=1}^{g+1}$ to $W^g$.\p

\begin{theorem}\label{geom2} For $0\le m\le g$, the restriction of $\til{j}$ to $F_{T,m}^{(1)}$ is a locally closed embedding, with image equal to 
\[\begin{matrix}
\til{X}_{T,m}^{(1)}\bsl(\til{X}_{T,m-1}^{(0)}\cup\til{X}_{T,m-1}^{(1)}),&\text{if }m\text{ is even.}\\
\til{X}_{T,m}^{(0)}\bsl(\til{X}_{T,m-1}^{(1)}\cup\til{X}_{T,m-1}^{(2)}),&\text{if }m\text{ is odd.}\\\end{matrix}\]

where for $F_{T,m}^{(2)}$, we replace, in the case $m$ is even, the two superscripts ${}^{(1)}$ by ${}^{(2)}$.
\end{theorem}\p

\begin{remark}\label{halfexact} If we relax the condition of $m$-exactness to require only $\{W^r\}_{r=1}^{g+1}$ to be $m$-excellent and $m$-general, then the image of $\til{j}|_{F_{T,m}^{(1)}}$ will be $\til{X}_{T,m}^{(1)}\bsl\til{X}_{T,m-1}^{(0)}$ in the even case and $\til{X}_{T,m}^{(0)}\bsl\til{X}_{T,m-1}^{(1)}$ in the odd case. However our definition of $m$-exactness is what one should use for orbital integrals on even ramified unitary groups in Section \ref{seccompute}.
\end{remark}\hp

\p
\section{Main computation}\label{seccompute}

In this section we have $n=2g+1$ except for a part of subsection \ref{submore}, where we'll state differently. We work with the assumption that $\text{char}(k)\gg0$, and leave it to Appendix \ref{char} to explain why this assumption may be dropped.\p

The nilpotent orbits $\CO\in\CO(0)$ of $\tLg(F)$ are classified as follows: The stable orbits, just like in $\mathfrak{gl}_{2g+1}$, are classified by partitions $\lambda=(\lambda_1^{\alpha_1}...\lambda_s^{\alpha_s})$ of $2g+1$ which give the sizes of the Jordan blocks, that is $ \lambda_1>...>\lambda_s$ and $\sum\alpha_i\lambda_i=2g+1$. In such a stable orbit, the orbits are classified by
\[\{(d_i)_{i=1}^s\,|\,\prod_{\lambda_i\text{ odd}} d_i=(-1)^gN_{E/F}E^{\times}\},\]

in which $d_i\in F^{\times}/N_{E/F}E^{\times}$ ($\cong\mu_2$) if $\alpha_i(\lambda_i-1)$ is even and $d_i\in \pi^{1/2}(F^{\times}/N_{E/F}E^{\times})$ (a torsor of $\mu_2$) if $\alpha_i(\lambda_i-1)$ is odd. We'll also denote by $(\lambda,(d_i)_{i=1}^s)_{\tLg}$ the corresponding nilpotent orbit in $\tLg$.\p

This classification goes as follows: let $\til{V}$ be the standard representation of $\til{G}/_E$, i.e. $\til{V}$ is a $(2g+1)$-dimensional hermitian space over $E$, with hermitian form $\inn$. Write $\til{N}\in\tLg$ an arbitrary element in such a nilpotent orbit with $\alpha_i$ Jordan blocks of sizes $\lambda_i$. There exists a unique decomposition $\til{V}=\bigoplus_{i=1}^s\til{V}_i$ such that $\til{N}$ preserves each $\til{V}_i$, that all Jordan blocks of $\til{N}|_{\til{V}_i}$ are of size $\lambda_i$, and that different $\til{V}_i$ and $\til{V}_j$ are orthogonal under $\inn$.\p

We have by definition $\til{N}^{\lambda_i-1}$ induces an isomorphism from $\til{V}_i/\til{N}(\til{V}_i)$ to $\ker(\til{N}|_{\til{V}_i})$. Also one has by the anti-hermitian property of $\til{N}$ that $\pair{\til{N}(\til{V}_i)}{\ker(\til{N}|_{\til{V}_i})}=0$. This allows us to consider a pairing on $\til{V}_i/N(\til{V}_i)$ by $\pair{\cdot}{\til{N}^{\lambda_i-1}\cdot}$. This pairing is non-degenerate, and it is hermitian if $\lambda_i$ is odd and anti-hermitian if $\lambda_i$ is even. The invariant $d_i$ is then the discriminant of this pairing.\p

%denote by $\til{\iota}:\{\pm1\}\xra{\sim}F^{\times}/N_{E/F}E^{\times}$ the isomorphism between the two groups. Let $\til{V}$ be the standard representation of $\til{G}/_E$, i.e. $\til{V}$ is a $(2g+1)$-dimensional hermitian space over $E$. For each $\lambda_i$, let $\til{V}_i'$ be the sum of all Jordan blocks of sizes $\lambda_i$. The restriction of the hermitian form $\inn$ on $\til{V}$ to $\til{V}_i'$ is non-degenerate. If $\lambda_i$ is odd, we assign $\til{\iota}(a_i)$ to be its discriminant.\p

%If $\lambda_i$ is even, let $\{\til{v}_1,...,\til{v}_{\alpha_i}\}$ be any lift of a basis of $\ker((\til{N}_m|_{\til{V}_i'})^{(\lambda_i/2+1)})/\ker((\til{N}_m|_{\til{V}_i'})^{\lambda_i/2})$. Then the hermitian form $\inn$ on $\til{V}_i'$ gives a non-degenerate hermitian form $\langle\cdot,\til{N}_m\cdot\rangle$ on the subquotient $\text{span}_E\{\til{v}_1,...,\til{v}_{\alpha_i}\}$. We then assign $\til{\iota}(a_i)$ to be the discriminant of this hermitian form on the subquotient.\p

Similarly, we can speak of nilpotent orbits in $\Lg(1)(k)$, i.e. $G(0)(k)$-orbit in $\Lg(1)(k)$ that are nilpotent in $\Lg$. The stable orbits correspond to the same partitions, and the orbits inside a stable orbit are classified by \[\{(d_i)_{i=1}^s\,|\,d_i\in k^{\times}/k^{\times2},\prod d_i^{\lambda_i}=(-1)^g\}.\]

The classification is done like above by replacing $\til{V}$ by $V$ (the standard representation of $G$), both hermitian and anti-hermitian forms by quadratic forms over $k$, and both $F^{\times}/N_{E/F}E^{\times}$, $\pi^{1/2}(F^{\times}/N_{E/F}E^{\times})$ by $k^{\times}/k^{\times2}$. It's not hard to check that this set is in bijection with the previous one. For our purpose we'd like to give a canonical bijection as follows: for any $N\in\Lg(1)(k)$ nilpotent, there exists a lift $\til{N}\in\tLg(F)_{\bx,-1/2}$ which is also nilpotent. The orbit of such $\til{N}$ is uniquely determined by the orbit of $N$.\p

Let $N_0\in\Lg(1)(k)$ be an arbitrary regular nilpotent element, i.e. one with a single Jordan block. For $1\le m\le g$, let $N_m\in\Lg(1)(k)$ be a nilpotent element with two Jordan blocks of sizes $2g+1-m$ and $m$ whose orbit is classified as $((2g+1-m,m),(-1)^g,1)_{\Lg(1)}$ if $m$ is even and $((2g+1-m,m),1,(-1)^g)_{\Lg(1)}$ if $m$ is odd. Write $\til{N}_m$ for the corresponding nilpotent orbit in $\tLg(F)$. For $m>0$ it's classified by $((2g+1-m,m),(-1)^g,\pi^{-1/2})_{\Lg(1)}$ if $m$ is even and $((2g+1-m,m),\pi^{-1/2},(-1)^g)_{\Lg(1)}$ if $m$ is odd.\p

%satisfying the following: let $2n$ be the even number in $\{m,2g+1-m\}$. Then $N_m$ restricts to a regular nilpotent operator on the $2n$-dimensional subspace $W\subset V$ defined by that Jordan block. The quadratic form on $V$ remains non-degenerate on $W$. Let $v\in W$ be such that $(N_m)^{2n-1}v\not=0$. We require $(-1)^{g+1}\langle v,(N_m)^{2n-1}v\rangle\in k^{\times2}$, and this uniquely determines the orbit of $N_m$. We also write $\til{N}_m\in\tLg$ for an arbitrary lift of $N_m$ in the sense of previous paragraph.\p

When $m>0$ there are always two orbits in the stable orbit of $\til{N}_m$. We again fix $T\in\Lg(1)^{rs}(k)$ and a lift $\til{T}\in\tLg(F)_{\bx,-1/2}$ in this section. We shall prove our main theorem (see Theorem \ref{intromain} for the definition of $a_m(T)$):

\begin{theorem}\label{main}
For $0\le m\le g$, we have $\Gamma_{\til{N}_m}^{st}(\til{T})=a_m(T)$ for any lift $\til{T}$ of $T$.\hp

Also $\Gamma_{\til{N}_m'}^{st}(\til{T})=(-1)^ma_m(T)$ for the other nilpotent orbit $\til{N}_m'$ in the same stable orbit.
\end{theorem}\p

For notational convenience, in this section we only compute the Shalika germs for $\til{N}_m$. For the other orbit the computation is identical except that we should replace $C_T$ by its quadratic twist, resulting in the sign $(-1)^m$ in the theorem.\p

We want to plug in (\ref{Shalika}) some test functions $f$ that are locally constant by a ``sufficiently large'' lattice and for which we know how to compute $J(\til{T},f)$. Let $S\subset B\subset G$ be a choice of $\theta$-stable maximal $k$-torus and Borel $k$-subgroup (see also the end of subsection \ref{elem}). Let $\Lb=\text{Lie }B$ and let $B(0)=B\cap G(0)$, $\Lb(i)=\Lb\cap\Lg(i)$, so that $\Lb=\Lb(0)\oplus\Lb(1)$. The same notations apply to $S$.\p

There exists a point $\mathbf{y}$ on the Bruhat-Tits building, which can be taken to be the barycenter of some alcove neighboring to $\bx$, such that $\tLg(F)_{\mathbf{y},-1/2}$ is the preimage of $\Lb(1)(k)$ under $\tLg(F)_{\bx,-1/2}\twoheadrightarrow\Lg(1)(k)$. We make the following hypothesis, which holds by \cite[Thm. 2.1.5]{De} when $\text{char}(k)$ is large enough (compared to $g$).\p

\begin{hypothesis}\label{homo}
Equality (\ref{Shalika}) holds for any compactly supported function that are locally constant by $\tLg(F)_{\mathbf{y},-1/2}$.
\end{hypothesis}\hp

Now we can choose our test functions. Fix now $0\le m\le g$ and let $N_m$ be as before. After a conjugation by some element in $G(0)(k)$, we may and shall assume that there exists a cocharacter $\rho_m:\Gm/_k\ra S(0)$ such that $\rho_m(\lambda)$ acts on $N_m$ by $\lambda^{-2}$. Write $\Lg_j\subset\Lg$ for the subspace on which $\rho_m(\lambda)$ acts by $\lambda^j$, $\Lg(1)_j=\Lg_j\cap\Lg(1)$ and $\Lg(1)_{\ge i}=\bigoplus_{j\ge i}\Lg(1)_{j}$. We may and shall assume that $\rho_m$ lies in the correct Weyl chamber so that $\Lb=\Lg(1)_{\ge 0}$.\p

Let $U$ be the unipotent radical of $B$. $U(0)=U\cap B(0)$ is the unipotent radical of $B(0)$. Let $f_m\in C_c^{\infty}(\tLg(F))$ be the function with support inside $\tLg(F)_{\bx,-1/2}$ defined by $f_m(X)=q^{-g^2}=(\#U(0)(k))^{-1}$ if the image of $X$ in $\Lg(1)$ is in the affine subspace $N_m+\Lg(1)_{\ge-1}$, and $f_m(X)=0$ otherwise. Theorem \ref{main} follows evidently from Hypothesis \ref{homo} and the following two propositions.\hp

\begin{proposition}\label{geomprop}For $0\le m\le g$,
\[J^{st}(\til{T},f_m)=\sum_{m'=0}^ma_{m'}(T)\left[
q^{\lfloor\frac{m-m'+1}{2}\rfloor}\binom{g-m'}{\lfloor\frac{m-m'}{2}\rfloor}
+\sum_{j=0}^{\lfloor\frac{m-m'}{2}\rfloor-1}(q^{m-m'-j}-q^{m-m'-j-1})\binom{g-m'}{j}\right].\]
\end{proposition}\p

Here $J^{st}(\til{T},f_m)$ is the sum of $J(\til{T}',f_m)$ where $\til{T}'$ runs over representatives of the orbits in the stable orbit of $\til{T}$.\p

\begin{proposition}\label{combprop}For $0\le m'\le m\le g$,
\[J(\til{N}_{m'},f_m)=q^{\lfloor\frac{m-m'+1}{2}\rfloor}\binom{g-m'}{\lfloor\frac{m-m'}{2}\rfloor}
+\sum_{j=0}^{\lfloor\frac{m-m'}{2}\rfloor-1}(q^{m-m'-j}-q^{m-m'-j-1})\binom{g-m'}{j}.
\]

For nilpotent orbits $\CO$ other than (the orbit of) $\til{N}_{m'}$ with $0\le m'\le m$, $J(\CO,f_m)=0$.
\end{proposition}\p

\subsection{Geometric identification}\label{secss} 

The goal in this subsection is to prove Proposition \ref{geomprop}. We begin with\p

\begin{lemma}\label{Galois} There is a natural bijection between $G(0)(k)$-orbits of $T$ in its stable orbit and $\tG(F)$-orbit of $\til{T}$ in its stable orbit.
\end{lemma}\hp

\begin{proof}For this proof only we'll replace $\tG$ by $SU_{2g+1}(E/F)$. One checks that this replacement does not affect the orbits. The orbits in the stable orbit of $T$ are classified by $\ker(H^1(k,\Stab_{G(0)}(T))\ra H^1(k,G(0))$ and that of $\til{T}$ by $\ker(H^1(F,\Stab_{\tG}(\til{T}))\ra H^1(F,\tG)$. By Lang's theorem and the fact that simply connected group over a non-archimedean local field has trivial $H^1$, we have $H^1(k,G(0))=H^1(F,\tG)=0$. Recall also that $\Stab_{G(0)}(T)\cong J_T[2]$.\p

The key is that our $\til{T}$ has its centralizer $\til{G}_{\til{T}}$ is anisotropic over $F^{ur}$, the maximal unramified extension of $F$ \cite[Thm 2.1]{Ts}. Consider the exact sequence
\mhp
\[1\ra H^1(\text{Gal}(F^{ur}/F),\Stab_{\tG}(\til{T})(F^{ur}))\ra H^1(F,\Stab_{\tG}(\til{T}) )\ra H^1(F^{ur},\Stab_{\tG}(\til{T}) )\]\hp

The last cohomology group is trivial by Steinberg's theorem. The first cohomology group is isomorphic to $H^1(k,J_T[2])$ because $J_T[2](\bar{k})$ is a quotient of $\Stab_{\tG}(\til{T})(F^{ur})$ with kernel possessing a filtration with graded pieces $\cong\mathbb{G}_a$.\p

This finishes the proof of the lemma. Note that from the exact sequence, one also sees that all orbits in the stable orbit of $\til{T}$ appear in $\tLg(F)_{\bx,-1/2}$, and the bijection just established is compatible with the reduction map $\tLg(F)_{\bx,-1/2}\twoheadrightarrow\Lg(1)(k)$ that sends $\til{T}\mapsto T$.
\end{proof}\hp

The main result in this subsection is to translate from the subsection \ref{subodd} that\p

\begin{lemma}\label{matrix}For $0\le m\le g$,\mhp

\[J(\til{T},f_m)=\frac{1}{\#J_T[2](k)}(\#\til{X}_{T,m}(k)-\#\til{X}_{T,m-1}(k)).\]
and
\[J^{st}(\til{T},f_m)=\#X_{T,m}(k)-\#X_{T,m-1}(k).\]
\end{lemma}\hp

\begin{proof}
To ease notation we deal with the case $m>0$. The proof applies to $m=0$ case with a little change in various places. By \cite[Thm 2.1]{Ts}, for $\til{h}\in\tG(F)$, $\Ad(\til{h})\til{T}\in\tLg(F)_{\bx,-1/2}$ iff $\til{h}\in\tG(F)_x$. In particular the centralizer $\Stab_{\tG}(\til{T})(F)\subset\tG(F)_x$. Moreover, $\tG(F)_{\bx,1/2}$ acts trivially on $f_m$ since $f_m$ is locally constant by $\tLg(F)_{\bx,0}\subset\tLg(F)_{y,-1/2}$. The integral is thus essentially a sum over $\tG(F)_x/\tG(F)_{\bx,1/2}\cong O_{2g+1}(k)$.\p

The measure of $\tG(F)_{\bx,1/2}$ is equal to that of $\tLg(F)_{\bx,1/2}$, which in Appendix \ref{norm} can be checked to be $q^{-\frac{2g^2+g}{2}}$. The measure of $\Stab_{\tG}(\til{T})(F)_{1/2}=\Stab_{\tG}(\til{T})(F)\cap \tG(F)_{\bx,1/2}$ is $1$. The image of $\Stab_{\tG}(\til{T})(F)$ in $\text{O}_{2g+1}(k)$ is equal to $\Stab_{\text{O}_{2g+1}(k)}(T)$, which has order $2\#J_T[2](k)$. Also $|D(\til{T})|=q^{2g^2+g}$. We thus have

\begin{equation}\label{ss}
J(\til{T},f_m)=\frac{1}{2\#J_T[2](k)}\sum_{\bar{h}\in \text{O}_{2g+1}(k)}f_m(\Ad(\bar{h})(T))=\frac{1}{\#J_T[2](k)}\sum_{\bar{h}\in \text{SO}_{2g+1}(k)}f_m(\Ad(\bar{h})(T)),
\end{equation}\hp

where $f_m$ in the RHS is understood as a function on $\tLg(F)_{\bx,-1/2}/\tLg(F)_{\bx,0}\cong\Lg(1)$.\p

We have $\rho_m$ acts on $V$ (the standard representation of $G(0)\cong\text{SO}_{2g+1}$) by weights $2g-m$, $2g-m-2$, ..., $m$, $m-1$, ..., $-m$, $-m-2$, ..., $-2g+m$. Let $V_g,...,V_{-g}\subset V$ be the $1$-dimensional subspace on which $\rho_m$ acts by scalars with corresponding weights (in order). Since the quadratic form $\inn$ on $V$ is preserved by $G(0)$, we have $V_j\subset V_i^{\perp}$ unless $i=-j$, i.e. unless their weights sum up to zero. Note that $\rho_m$ acts on $N_m$ with weight $-2$. This implies $N_m(V_i)=V_{i-1}$ for $m<i\le g$ and $-g<i\le -m$, and that $N_m(V_{i+1})=V_{i-1}$ for $-m<i<m$.\p

We also write $V_{\ge n}:=\bigoplus_{n\le i\le g}V_i$. From the description of $N_m$ above, one sees that if $\Ad(h)T\in N_m+\Lg(1)_{\ge -1}$ for some $h\in G(0)(k)$, then there exists $W^g\subset V\oplus k$ such that $\pi_1(W^g)=h^{-1}V_{\ge 1}$, and the flag $\left(0\subset h^{-1}V_g\subset h^{-1}V_{\ge g-1}\subset...\subset h^{-1}V_{\ge 2}\subset W^g\right)$ is $m$-exact (see the definition of $m$-exactness in the paragraph before Theorem \ref{geom}).\p

In fact, by the definition of $N_m$, there exists $v_1\in V_1$ be such that $\langle v_1, N_m(v_1)\rangle=1$. With it $W^g$ can be given by any of the following two choices $W^g=h^{-1}V_{\ge 2}+(h^{-1}.v_1,\pm1)$, where $h^{-1}V_{\ge 2}$ is a subspace of $V$ and thus of $V\oplus k$.\p

Conversely, if there exists an $m$-exact flag $0\subset W^1\subset...\subset W^g\subset V\oplus k$, then there is a unique right $B(0)(k)$-coset, say $B(0)(k)\cdot h$, such that $W^i=h^{-1}V_{\ge g-i+1}$ for $1\le i<g$ and $\pi_1(W^g)=h^{-1}V_{\ge1}$. In this right coset, one checks from the definition of $m$-exactness that there are exactly two right $U(0)(k)$-cosets, say $U(0)(k)\cdot h$, such that $\Ad(h)T\in N_m+\Lg(1)_{\ge-1}$.\p

In other words, there is a 2-2 correspondence between such $m$-exact flags and right $U(0)(k)$-cosets $U(0)(k)\cdot h$ satisfying $\Ad(h)T\in N_m+\Lg(1)_{\ge-1}$. Since $f_m(X)=q^{-g^2}$ when $X\in N_m+\Lg(1)_{\ge-1}$ (recall $q^{-g^2}=(\#U(0)(k))^{-1}$), we conclude from (\ref{ss}) that
\[J(\til{T},f_m)=\frac{1}{\#J_T[2](k)}\#F_{T,m}(k)=\frac{1}{\#J_T[2](k)}\left(\til{X}_{T,m}(k)-\til{X}_{T,m-1}(k)\right).
\]\hp

The last equality follows from Theorem \ref{geom}. This proves the first statement of the lemma.\p

For the stable Shalika germ, by Lemma \ref{Galois} we have $T$ running over $G(0)$-orbits in its stable orbit, classified by $H^1(k,J_T[2])$. When $T$ runs over these orbits, above every $k$-point of $X_{T,m}$, all isomorphism classes of $J_T[2]$-torsor will appear exactly once in $\til{X}_{T,m}$. Since $\#H^1(k,J_T[2])=\#H^0(k,J_T[2])=\#J_T[2](k)$, the sum of the number of $k$-points in all isomorphism classes is exactly $\#J_T[2](k)$. This gives
\[J^{st}(\til{T},f_m)=\#X_{T,m}(k)-\#X_{T,m-1}(k).\qedhere\]
\end{proof}\p

\begin{lemma}\label{two}(i)$\;$ $\#X_{T,m}(k)=\#\Sym^m(C_T)(k)-q\#\Sym^{m-2}(C_T)(k)$.\hp

For $0\le m\le g$, write $\widehat{a}_m(T)=(-1)^m\text{Tr}(\text{Frob}:H^m(J_T/_{\bar{k}},\Ql))$. Then\qp

(ii)$\;\,$ $\displaystyle\#\Sym^m(C_T)(k)=\sum_{m'=0}^m(q^{m'}+...+q+1)\widehat{a}_{m-m'}(T).$

(iii)$\;$ $\displaystyle \widehat{a}_m(T)=\sum_{m'=0}^{\lfloor\frac{m}{2}\rfloor}q^{m'}\binom{g-m+2m'}{m'}a_{m-2m'}(T)$.
\end{lemma}\hp

Here (i) comes from properties of hyperelliptic curves; above any rational point $D-(m-1)(\infty)\in X_{T,m}(k)\subset\text{Pic}^1(C_T)$, there will be $q^d+q^{d-1}+...+1$ rational points on $\Sym^m(C_T)$ and $q^{d-1}+...+1$ rational points on $\Sym^{m-2}(C_T)$, where $d=\ell(D):=H^0(C_T,D)-1$ is the dimension of the linear system. (ii) and (iii) are direct applications of the Grothendieck-Lefschetz fixed point formula and basic properties of $\ell$-adic cohomology of curves, their symmetric powers and Jacobians.\p

It is now straightforward to verify that Proposition \ref{geomprop} follows from Lemma \ref{matrix} and Lemma \ref{two}.\p

\subsection{Nilpotent affine Springer fiber}\label{secnil} We now prove Proposition \ref{combprop}. Counting points on nilpotent affine Springer fibers\footnote{I just make up this term to refer to nilpotent $p$-adic orbital integral. However nilpotent Springer fibers do mean Springer fibers over nilpotent elements.}, just like counting points on nilpotent Springer fibers, is usually purely combinatorial. Let $B^{\op}\subset G$ be the opposite Borel to $B$ with respect to $S$. Our essential idea here is that after applying the formula of Ranga Rao \cite[Thm 1]{RR}, we can do ``reduction modulo $\pi$'' and arrive at an integral over $G(0)(k)$ which is left invariant by $U(0)(k)$ and right invariant by $B^{\op}(0)(k)$. This gives a combinatorial sum over $U(0)(k)\bsl G(0)(k)/B^{\op}(0)(k)$, which is identified with the Weyl group of $G(0)$.\p

To begin our proof, the formula of Ranga Rao in our case can be formulated as follows. There exists a maximal $F$-split torus $\tS\subset\tG$, whose corresponding apartment contains $\bx$ and whose reduction at $\bx$ is equal to $S(0)\subset G(0)$. Moreover after conjugation we may assume that the cocharacter $\rho_{m'}:\Gm/_k\ra S(0)$ corresponds to $\til{\rho}_{m'}:\Gm/_F\ra\tS$ and that $\til{\rho}_{m'}$ also acts on $\til{N}_m'$ by weight $-2$.\p

Fix such an $\tS$ and $\til{\rho}_{m'}$. Denote by $\CO$ in this subsection the orbit of $\til{N}_{m'}$. Write $\tLg_i\subset\tLg$ be the subspace on which $\til{\rho}_{m'}$ acts by weight $i$. Then with suitably normalized measure, Ranga Rao's formula says
\begin{equation}\label{RR}
J(\til{N}_{m'},f_m)=\int_{\tLg_{\le -2}(F)\cap\CO}\varphi(\til{X})\int_{\tG(F)_{\bx,0}}f_m(\Ad(\til{h})\til{X})d\til{h}d\til{X},
\end{equation}

where $\varphi(\til{X})$ is an $\mathbb{R}$-valued function on $\tLg_{-2}(F)$ and the measure on the first integral is a Haar measure on $\tLg_{\le -2}$. The space $\tLg_{-2}$ can be interpreted as follows: let $\til{V}$ be the standard representation of $\til{G}/_E$, i.e. $\til{V}$ is a $(2g+1)$-dimensional hermitian space over $E$. Then $\til{\rho}_{m'}$ acts on $\til{V}$ with weights $(2g-m'),(2g-m'-2),...,m',(m'-1),...,1,0,-1,...,-m',(-m'-2),...,(-2g+m')$.\p

Denote by $\til{V}_g,...,\til{V}_{-g}$ the $1$-dimensional $E$-subspace with these weights, respectively. Let $\tLg_{ij}$, $-g\le i,j\le g$ be the $1$-dimensional $E$-subspace of $\tLg$ which maps $\til{V}_j$ to $\til{V}_i$. One can then check 
\[\tLg_{-2}=\bigoplus_{j-i=1,|i+j|>2m'}\tLg_{i,j}\oplus\bigoplus_{j-i=2,|i+j|<2m'}\tLg_{i,j}\]
and
\[\tLg_{<-2}=\bigoplus_{j-i=2,|i+j|\ge 2m'}\tLg_{i,j}\oplus\bigoplus_{j-i\ge 3}\tLg_{i,j}.\]\p

Note that $\tLg_{ij}$ is not defined over $F$ unless $i+j=0$, but $\tLg_{ij}+\tLg_{-j,-i}$ is always defined over $F$. Now we fix a ``valuation-preserving'' identification of $\til{u}_{ij}:\tLg_{ij}\ra\mathbb{G}_a/_E$ so that $\til{u}_{ij}^{-1}(\pi^{-1/2})$ is not in $\tLg_{ij}\cap\tLg(E)_{\bx,0}$ but $\til{u}_{ij}^{-1}(1)$ is. Let $|\cdot|:E\ra\mathbb{R}$ be the extension of the standard norm on $F$, i.e. $|\pi^{-1/2}|=q^{1/2}$. One then computes
\[\varphi(\til{X})=\prod_{j-i=2,\,j\equiv m'(2),\,|i+j|\le 2m'}|\til{u}_{ij}(\til{X})|,\; \til{X}\in\tLg_{\le -2}(F).\]\p

In (\ref{RR}), if $\til{X}\in\tLg_{\le -2}$ is such that $|\til{u}_{ij}(\til{X})|>q^{1/2}$, then $\til{X}\not\in\tLg(F)_{\bx,-1/2}$ and $f_m(\Ad(\til{h})(\til{X}))=0$. Moreover, the value of $f_m(\Ad(\til{h})(\til{X}))$ depends only on $\til{u}_{ij}(\til{X})$ modulo $\CO_E$ since $f_m$ is locally constant by $\tLg(F)_{\bx,0}$. Let's now denote by
$Z_{m'}$ the image of $\tLg(F)_{\bx,-1/2}\cap\tLg_{\le -2}(F)\cap\CO$ in $\Lg(1)(k)$
and by $(d\mu)_{m'}$ the push-forward of the measure $\varphi(\til{X})d\til{X}|_{\tLg_{\le-2}(F)\cap\tLg(F)_{\bx,-1/2}}$ to the finite set $Z_{m'}$. Also let $dh$ be the push-forward of the measure $d\til{h}$ from $\tG(F)_{\bx,0}$ to $G(0)(k)$. Then we can rewrite (\ref{RR})
\begin{equation}\label{RR2}
J(\til{N_m'},f_m)=\int_{X\in Z_{m'}}\int_{G(0)(k)}f_m(\Ad(h)X)dh\cdot(d\mu)_{m'},
\end{equation}\hp

We can similarly define $V_g,...,V_{-g}$ as $1$-dimensional $k$-subspace on which $\rho_{m'}$ acts by strictly decreasing weight. In fact $V_i$ is just the line spanned by $v_i$ in the previous subsection (with $m$ replaced by $m'$). We can then define $u_{ij}:\Lg(k)\ra k$ in the same way, scaled so that if we write the reduction maps $\text{red}_1:\tLg(F)_{\bx,-1/2}\ra\Lg(1)$ and $\text{red}_2:\pi^{-1/2}\CO_E\ra k$, then $u_{ij}(\text{red}_1(\til{X}))=\text{red}_2(\til{u}_{ij}(\til{X}))$ for $\til{X}\in\tLg(F)_{\bx,-1/2}$.\p

\begin{lemma}\label{Sg}If $f_m(\Ad(h)(X))\not=0$ for $X\in Z_m'$, then $u_{ij}(X)\not=0$ when $j-i=2$ and $|i+j|<2m'$.
\end{lemma}\hp

\begin{proof}The function $f_m(X)$, as a function of $X\in\Lg(1)(k)$, is invariant under conjugation by $U(0)(k)$. The assertion of the lemma, for $X\in Z_m'$, is a property that is preserved under conjugation by $B^{\op}(0)(k)$. Therefore it suffices to consider $h$ in a set of representative for $U(0)(k)\bsl G(0)(k)/B^{\op}(0)(k)$, which can be taken to be the Weyl group $N_{G(0)(k)}(S(0)(k))/S(0)(k)$.\p

Identify $S_g\ltimes\{\pm1\}^g$ with $N_{G(0)(k)}(S(0)(k))/S(0)(k)$ in the following way: the first component $S_g$ shall permute $V_g,...,V_1$, and the $i$-th $\{\pm1\}$ in the second component switchs $V_i$ and $V_{-i}$. We now check directly the assertions for all $\sigma\in S_g\ltimes\{\pm1\}^g$. To have $f_m(\Ad(\sigma^{-1})X)\not=0$ for some $X\in Z_{m'}$, it's necessary that $\Ad(\sigma)N_m\in Z_{m'}$. This happens exactly when\hp

\begin{condition}\label{cond} Consider $\sigma\in S_g\ltimes\{\pm1\}^g$ acting on $\{0,\pm1,...,\pm g\}$ where $S_g$ permutes $\{1,...,g\}$ and $\{-1,...,-g\}$ simultaneously, the $i$-th component in $\{\pm1\}^g$ switches $\pm i$, and $0$ is always fixed. Now for any $-g\le i<j\le g$,\hp
(i)$\;\;$ For $j-i=2$, $|i+j|<2m$, either $\sigma(j)-\sigma(i)=1$ and $|\sigma(i)+\sigma(j)|>2m'$, or $\sigma(j)-\sigma(i)\ge 2$.\qp
(ii)$\;\;\,\,$ For $j-i=1$, $2m<|i+j|\le 2g+1$, the same condition is required.
\end{condition}\hp

It's straightforward to see that the condition is satisfied only when $\sigma\in S_g$, i.e. $\sigma$ preserves $\{1,...,g\}$. One then see inductively that $\sigma^{-1}(0)=0\Rightarrow\sigma^{-1}(1)=1\Rightarrow\sigma^{-1}(2)=2\Rightarrow...$, until $\sigma^{-1}(m')=m'$. We conclude that $\sigma$ and thus $\sigma^{-1}$ preserves $V_1,...,V_{m'}$. Since $u_{i-1,i+1}(N_m)\not=0$ for $i=0,...,m'-1$, this implies that $u_{i-1,i+1}(X)\not=0$ for the same $i$'s, which is what we have to prove.
\end{proof}\hp

On the subset $Z_{m'}^o\subset Z_{m'}$ where the conclusion of the lemma holds, $(d\mu)_{m'}$ is nothing but a multiple of the counting measure. We'll pretend it is exactly the counting measure and discuss the normalization constant later. The idea in the lemma can now be further applied to compute the integral; write
\[I_{m'}(h):=\sum_{X\in Z_{m'}^o}f_m(\Ad(h)X).\]\qp

We have to compute $\displaystyle\sum_{h\in G(0)(k)}I_{m'}(h)$ (up to a normalizing constant). Exactly as in the situation of the previous lemma, This function $I_{m'}$ is invariant under left translation by $U(0)(k)$ and right translation by $B^{\op}(0)(k)$, and we arrive at a sum over the Weyl group $S_g\ltimes\{\pm1\}^g$. To have $I_{m'}(\sigma^{-1})\not=0$, $\sigma$ needs to satisfy Condition \ref{cond}.\p

Denote by $\Xi_{m,m'}\subset S_g$ the set of such $\sigma$. For $\sigma\in\Xi_{m,m'}$, the number of elements in the double coset $U(0)(k)\sigma^{-1}B^{\op}(0)(k)$ is given by $\#B(0)(k)\cdot q^{\delta_1(\sigma)}$. Also the sum $I_{m'}(\sigma^{-1})=q^{\delta_2(\sigma)-g^2}$, where
\[\delta_1(\sigma)=g^2-\#\{1\le i<j\le g\,|\,\sigma(i)>\sigma(j)\}.\]
\[\delta_2(\sigma)=m-m'+\#\{1\le i<j\le g\,|\,\sigma(i)>\sigma(j),\;j-i>1\}.\]\p

We also have to figure out the normalization of measures. The single choice $X=N_{m'}\in Z_{m'}^o\subset\Lg(1)(k)$ and $h=id\in G(0)(k)$ correspond to the lattice $\tLg(F)_{\bx,1/2}/(\tLg(F)_{\bx,1/2}\cap\tLg_{\til{N}_{m'}})\subset\tLg(F)/\tLg_{\til{N}_{m'}}(F)$. A careful inspection of the normalization at the end of Appendix \ref{norm} shows that this lattice is to have measure $1$.\p

In addition, as Ranga Rao's method begins with Iwasawa decomposition $\tG=\tB\cdot\tG(F)_{\bx,0}$, we also have to divide by how much they intersect, namely the order of $(\tG(F)_{\bx,0}\cap\tB)\tG(F)_{\bx,1/2}/\tG(F)_{\bx,1/2}$, which is $\#B(0)(k)$. In summary, all the way from (\ref{RR2}) we have
\begin{equation}\label{RR3}
J(\til{N}_{m'},f_m)=\sum_{\sigma\in\Xi_{m,m'}}q^{\delta_1+\delta_2-g^2}
\;\;\;\;\;\;\;\;\;\;\;\;\;\;\;\;\;\;\;\;\;\;\;\;\;\;\;\;\;\;\;\;\;\;\;\;\;\;
\;\;\;\;\;\;\;\;
\end{equation}
\mhp
\[=\sum_{\sigma\in\Xi_{m,m'}}q^{m-m'-\#\{1\le i<g\,|\,\sigma(i)>\sigma(i+1)\}}.\]\hp

Proposition \ref{combprop} now follows from Proposition \ref{path}. This completes the proof of Theorem \ref{main}.\p

\subsection{More Shalika germs}\label{submore} Recall we had (in the beginning of Section \ref{secgeom}) $j_m:\Sym^m(C_T)\ra\text{Pic}^1(C_T)$ and $\times2:F_T\ra\text{Pic}^1(C_T)$. The latter map is \'{e}tale Galois with Galois group $J_T[2]$. Denote by $\widetilde{\Sym}^m(C_T):=\Sym^m(C_T)\times_{\text{Pic}^1(C_T)}F_T$ the fiber product, which is an \'{e}tale $J_T[2]$-cover of $\Sym^m(C_T)$. We have\p

\begin{theorem}\label{tilded}For $0\le m\le g$, we have
\[\Gamma_{\til{N}_m}(\til{T})=\frac{1}{\#J_T[2](k)}\left(\sum_{0\le2\ell\le m}\#\widetilde{\Sym}^{m-2\ell}(C_T)(k)\cdot q^{\ell}\cdot C_{\ell}(-g+m-2\ell+1)\right.\]
\mhp
\[\left.-(q+1)\sum_{0<2\ell+1\le m}\#\widetilde{\Sym}^{m-2\ell-1}(C_T)(k)\cdot q^{\ell}\cdot C_{\ell}(-g+m-2\ell)\right).\]\qp

See Definition \ref{Cata} for the combinatorial numbers $C_{\ell}(\cdot)$. When $m>0$, for the other nilpotent orbit in the stable orbit of $\til{N}_m$, simply change $C_T$ to its quadratic twist. 
\end{theorem}\hp

\begin{proof}Write in this proof
\[\vec{u}={\small\left(\begin{matrix}\Gamma_{\til{N}_0}(\til{T})\\\Gamma_{\til{N}_1}(\til{T})\\\Gamma_{\til{N}_2}(\til{T})\\...\end{matrix}\right)},\;\vec{v}={\small\left(\begin{matrix}J(\til{T},f_0)\\J(\til{T},f_1)\\J(\til{T},f_2)\\...\end{matrix}\right)},\;\vec{w}={\small\left(\begin{matrix}\#\widetilde{\Sym}^0(C_T)(k)\\\#\widetilde{\Sym}^1(C_T)(k)\\\#\widetilde{\Sym}^2(C_T)(k)\\...\end{matrix}\right)}.\]\p

We have to write $\vec{u}$ in terms of $\vec{w}$. The first half of Lemma \ref{matrix} says
\[J_{\til{T}}(f_m)=\frac{1}{\#J_T[2](k)}\left(\#\til{X}_{T,m}(k)-\#\til{X}_{T,m-1}(k)\right)\]
and a $J_T[2]$-cover version of Lemma \ref{two}(i) gives:
\[\#\til{X}_{T,m}(k)=\#\widetilde{\Sym}^m(C_T)(k)-q\cdot\widetilde{\Sym}^{m-2}(C_T)(k).\]\p

Putting together, they imply $\vec{v}=B^{(1)}B^{(2)}\vec{w}$, where $(B^{(1)}_{ij})_{0\le i,j\le g}$ and $(B^{(2)}_{ij})_{0\le i,j\le g}$ are lower triangular matrices with
\[B^{(1)}_{ij}=\left\{\begin{matrix}
1&\text{if }i=j.\\
-1&\text{if }i=j+1.\\
0&\text{otherwise}.
\end{matrix}\right.,\;B^{(2)}_{ij}=\left\{\begin{matrix}
1&\text{if }i=j.\\
-q&\text{if }i=j+2.\\
0&\text{otherwise}.
\end{matrix}\right.,\;\]\p

To recover $\Gamma_{\til{N}_m}(\til{T})$ from $J_{\til{T}}(f_m)$, i.e. compute $\vec{u}$ in terms of $\vec{v}$, we need to ``invert'' Proposition \ref{combprop}. One observe that Proposition \ref{combprop} is the same as saying $\vec{v}=B^{(3)}B^{(2)}A\vec{u}$, where $A$ is the matrix in Proposition \ref{inverse}, with $x=g$, and
\[(B^{(3)})_{ij}=\left\{\begin{matrix}
q^{i-j}&\text{if }i\ge j\\[5pt]
0&\text{otherwise}.
\end{matrix}\right.\]\p

Proposition \ref{inverse} says
\[(A^{-1})_{ij}=\left\{\begin{matrix}
q^{\ell}C_{\ell}(-g+j)&\text{if }i=j+2\ell,\;\ell\in\Z_{\ge 0}\\[5pt]
0&\text{otherwise}.
\end{matrix}\right.\]\p

And we have $\vec{u}=A^{-1}(B^{(2)})^{-1}(B^{(3)})^{-1}B^{(1)}B^{(2)}\vec{w}$. One observe that $B^{(1)}$, $B^{(2)}$ and $B^{(3)}$ all commutes, and $B^{(4)}:=(B^{(3)})^{-1}(B^{(1)})$ is given by
\[(B^{(4)})_{ij}=\left\{\begin{matrix}
1&\text{if }i=j\\
-(q+1)&\text{if }i=j+1\\
q&\text{if }i=j+2\\
0&\text{otherwise}.
\end{matrix}\right.\]\hp

We have $\vec{u}=A^{-1}B^{(4)}\vec{w}$. By the formula for $A^{-1}$ and $B^{(4)}$, we see that if $i=j+2\ell$, $\ell\in\Z_{\ge0}$, then $(A^{-1}B^{(4)})_{ij}=q^{\ell}C_{\ell}(-g+j)+q^{\ell}C_{\ell-1}(-g+j+2)=q^{\ell}C_{\ell}(-g+j+1)$ by Proposition \ref{recur}, with $C_{-1}(\cdot)$ understood to be zero. And if $i=j+2\ell+1$, $\ell\in\Z_{\ge0}$, then $(A^{-1}B^{(4)})_{ij}=-(q+1)q^{\ell}C_{\ell}(-g+j+1)$. This is essentially the content of Theorem \ref{tilded}.
\end{proof}\hp

Another way of thinking of these covers $\widetilde{Sym}^m(C_T)$ is as follows. Denote by $\alpha_T$ the $J_T[2]$-torsor $(\times2)^{-1}(\infty)$. Then $\widetilde{Sym}^m(C_T)$ is the \'{e}tale $J_T[2]$-cover of $\Sym^m(C_T)$ for which the fiber above $m(\infty)$ is isomorphic to $\alpha_T$. There is a unique such one since $J_T[2]$ is the maximal abelian $2$-annihilated quotient of $\pi_{1}^{\acute{e}t,tame}(\Sym^m(C_T))$.\p

In the rest of this section we suppose $n=2g+2$. So that $\til{G}=U_{2g+2}(E/F)$ is instead an even quasi-split unitary group (still ramified). Recall we can take $\til{T}\in\tLg(F)_{\bx,-1/2}$ to be any lift of $T\in\tLg(F)_{\bx,-1/2}/\tLg(F)_{\bx,0}$, which is regular semisimple and is associated the genus $g$ (projective smooth) hyperelliptic curve $C_T=(y^2=(-1)^{g+1}p_T(x))$ where $p_T(x)$ is the degree $2g+2$ monic characteristic polynomial of $T$.\p

We use then Theorem \ref{geom2} instead of Theorem \ref{geom} to obtain the following result. Let $\lambda_1,\lambda_1',...,\lambda_g,\lambda_g'$ again be the Frobenius eigenvalues on $H^1(C_T/_{\bar{k}},\Ql)$ so that $\lambda_i\lambda_i'=q$. We also put artificially that $\lambda_0=1$, $\lambda_0'=q$. Write this time (note the difference on the range of $S$ with the odd case)
\mhp
\[a_m(T):=(-1)^m\cdot\sum_{S\subset\{0,...,g\},|S|=m}\left(\prod_{i\in S}(\lambda_i+\lambda_i')\right).\]

For $0\le m\le g+1$, let $\til{N}_m$ be any element in any nilpotent orbit in $\tLg=\text{Lie }\til{G}$ with two Jordan blocks of sizes $2g+2-m$ and $m$ (a regular nilpotent if $m=0$). There can be either $1$, $2$ or $4$ of such orbits.\p

%When $m>0$ is even, such an orbit is classified as $((2g+2-m,m),a_1,a_2)$ with $a_1,a_2\in\pi^{1/2}F^{\times}/N_{E/F}E^{\times}$ (resp. $a\in F^{\times}/N_{E/F}E^{\times}$ if $m=g+1$ and $g$ odd) as in the beginning of this section.\p

%We say such an orbit is hyperbolic if $a_1a_2=1$ (resp. $a=1$), and elliptic otherwise. Hyperbolic nilpotent orbits (with two Jordan blocks of sizes even) are characterized by that they're contained in the closure of one (and all) subregular nilpotent orbit, while elliptic ones are those that are not.\p

What we can show in parallel to Theorem \ref{main}, using the method in this section, is 

\begin{theorem}\label{even}
The stable Shalika germs at $\til{T}$ for nilpotent orbits with two Jordan blocks of an even quasi-split ramified unitary groups is
\[\Gamma_{\til{N}_m}^{st}(\til{T})=a_m(T).\]
\end{theorem}\p

As for general (non-stable) Shalika germs, we encountered a technical difficulty: what was developed in Theorem \ref{geom2} only allows us to compute Shalika germs for nilpotent $\tG^{ad}(F)$-orbits, where $\tG^{ad}=PU_{2g+2}(E/F)$. The image of $\tG(F)\ra\tG^{ad}(F)$ has index $2$ in $\tG^{ad}(F)$. If we take $u\in\tG^{ad}(F)$ to be any element outside the image, then what we can compute is the sum of Shalika germs $\Gamma_{\til{N}_m}(\til{T})+\Gamma_{\text{ad}(u)\til{N}_m}(\til{T})=\Gamma_{\til{N}_m}(\til{T})+\Gamma_{\til{N}_m}(\text{ad}(u)\til{T})$.\p

To state what we are able to obtain in parallel with Theorem \ref{tilded} with the geometry from Theorem \ref{geom2}, we need a notion about nilpotent orbits of $\tG=U_{2g+2}(E/F)$ with two even Jordan blocks.\p

\begin{definition}\label{ell}
Let $((2g+2-m,m),d_1,d_2)$ for $0<m<g+1$ with $m$ even (resp. $((g+1,g+1),d)$ for $m=g+1$ if $g+1$ is even) be a nilpotent orbit with two even Jordan blocks. We say the orbit is {\bf hyperbolic} if $d_1d_2=-1$ (resp. $d=1$), and {\bf elliptic} otherwise.\hp

We also say any nilpotent orbit with two odd Jordan blocks is hyperbolic. They are characterized by the following: for any two distinct nilpotent orbits, both having two Jordan blocks, one lies in the closure of the other if and only if they have different dimensions and they are either both hyperbolic or both elliptic.
\end{definition}\p

As in Theorem \ref{tilded}, we also need notations about covers of $C_T$. Recall we have two rational points $\infty^{(1)},\infty^{(2)}\in C_T(k)$ (see subsection \ref{subeven}). Fix a choice of any of them, say $\infty^{(1)}$. Consider $(\times2)^{-1}(\infty^{(1)})$, where $\times2$ is the \'{e}tale $J_T[2]$-Galois map in Theorem \ref{Jerry2}. This is a $J_T[2]$-torsor, which we shall denote by $\alpha_T$.\p

Consider also $\infty^{(1)}-\infty^{(2)}\in J_T(k)$. We have, by Lang's theorem, $J_T(k)/2J_T(k)\cong H^1(k,J_T[2])$. Denote by $\beta_T$ the $J_T[2]$-torsor that are given by $\infty^{(1)}-\infty^{(2)}$ in this way\footnote{In fact, $(\infty^{(1)})-(\infty^{(2)})\in 2J_T(k)$ except when all irreducible factors of $p_T(x)\in k[x]$ are even and $n=2g+2$ is divisible by $4$. Consequently, if there is an odd factor of $p_T(x)$ or if $g$ is even, $\beta_T$ is trivial.}. For even non-negative integers $m$, we write $\widetilde{\Sym}^m(C_T)$ the \'{e}tale $J_T[2]$-cover of $\Sym^m(C_T)$ for which the fiber above $\frac{m}{2}(\infty^{(1)})+\frac{m}{2}(\infty^{(2)})$ is (as a $J_T[2]$-torsor) isomorphic to $\alpha_T$. Write also $\widetilde{\Sym}^{m,*}(C_T)$ the \'{e}tale $J_T[2]$-cover of $\Sym^m(C_T)$ for which the fiber above $\frac{m}{2}(\infty^{(1)})+\frac{m}{2}(\infty^{(2)})$ is isomorphic to $\alpha_T\times^{J_T[2]}\beta_T$.\p

For odd $m$ instead, we write $\widetilde{\Sym}^m(C_T)$ the \'{e}tale $J_T[2]$-cover of $\Sym^m(C_T)$ for which the fiber above $\frac{m+1}{2}(\infty^{(1)})+\frac{m-1}{2}(\infty^{(2)})$ is isomorphic to $\alpha_T$. And we write $\widetilde{\Sym}^{m,*}(C_T)$ the \'{e}tale $J_T[2]$-cover of $\Sym^m(C_T)$ for which the fiber above $\frac{m+1}{2}(\infty^{(1)})+\frac{m-1}{2}(\infty^{(2)})$ is isomorphic to $\alpha_T\times^{J_T[2]}\beta_T$. For all $m$, we write $\til{S}^m(C_T)=\left(\#\widetilde{\Sym}^{m}(C_T)(k)+\#\widetilde{\Sym}^{m,*}(C_T)(k)\right)$.\p

Lastly, we write $C_T'$ to be the quadratic twists of $C_T$, so it has two points above infinity $\infty^{(1)}$, $\infty^{(2)}$ that are not defined over $k$. Note $\text{Pic}^0(C_T')[2]\cong\text{Pic}^0(C_T)[2]=J_T[2]$. For $m$ even, write $\widetilde{\Sym}^m(C_T')$ the \'{e}tale $J_T[2]$-cover of $\Sym^m(C_T')$ for which the fiber above $\frac{m}{2}(\infty^{(1)})+\frac{m}{2}(\infty^{(2)})$ is isomorphic to $\alpha_T$. We also write $\widetilde{\Sym}^{m,*}(C_T')$ the the \'{e}tale $J_T[2]$-cover of $\Sym^m(C_T')$ for which the fiber above $\frac{m}{2}(\infty^{(1)})+\frac{m}{2}(\infty^{(2)})$ is isomorphic to $\alpha_T\times^{J_T[2]}\beta_T$. And we write $\til{S}^m(C_T')=\left(\#\widetilde{\Sym}^{m}(C_T')(k)+\#\widetilde{\Sym}^{m,*}(C_T')(k)\right)$.\p

Note $\#(\text{Res}_k^{k[x]/p_T(x)}\mu_2)(k)=2^r$ where $r$ is the number of irreducible factors of $p_T(x)$ in $k[x]$. We have\p

\begin{theorem}\label{tilded2}
For $0\le m\le g+1$, let $\til{N}_m$ be any nilpotent orbit with two Jordan blocks of sizes $2g+2-m$ and $m$. Recall $u\in\tG^{ad}(F)$ is any element that doesn't come from $\tG(F)$. We have\p

If $\til{N}_m$ is hyperbolic, then
\[\Gamma_{\til{N}_m}(\til{T})+\Gamma_{\til{N}_m}(\text{ad}(u)\til{T})=2^{-(r-1)}\cdot\]
\mhp\mhp
\[\left(\sum_{0\le 2\ell\le m}\til{S}^{m-2\ell}(C_T)\cdot q^{\ell}\cdot \left(C_{\ell}(-g+m-2\ell+1)-(\sqrt{q}+\frac{1}{\sqrt{q}})^2C_{\ell-1}(-g+m-2\ell+1)\right)\right.\]
\mhp
\[\left.-2(q+1)\sum_{0<2\ell+1\le m}\til{S}^{m-2\ell-1}(C_T)\cdot q^{\ell}\cdot C_{\ell}(-g+m-2\ell)\right).\]\p

If otherwise $\til{N}_m$ is elliptic, then
\[\Gamma_{\til{N}_m}(\til{T})+\Gamma_{\til{N}_m}(\text{ad}(u)\til{T})=2^{-(r-1)}\;\cdot\]
\mhp\mhp
\[\left(\sum_{0\le 2\ell\le m}\til{S}^{m-2\ell}(C_T')\cdot q^{\ell}\cdot\left(C_{\ell}(-g+m-2\ell+1)+(\sqrt{q}+\frac{1}{\sqrt{q}})^2C_{\ell-1}(-g+m-2\ell+1)\right)\right).\]\hp

Here we adapt the convention that $\widetilde{\Sym}^{-1}=\widetilde{\Sym}^{-2}=\emptyset$. For case $m=0$ the two formulas agree.
\end{theorem}\hp

\hp
\section{Endoscopic transfer of nilpotent orbits for ramified unitary groups}\label{secendo}

In the beginning of this section and subsection \ref{subendo}, i.e. except in subsection \ref{substab}, we'll assume $\text{char}(F)=0$ so that the endoscopic transfer of Langlands-Shelstad \cite{LS} is valid. We however note that we can actually also work with sufficiently large $\text{char}(F)$ (in an un-effective manner) thanks to Gordan-Hales \cite{GH}.\p

In \cite{Assem}, Assem stated a conjecture regarding stable distributions supported on the nilpotent cone for a reductive $p$-adic group. Recall that a distribution $\mathcal{D}\in C_c^{\infty}(\tLg(F))^*$ is called {\bf stable} if $\mathcal{D}(f)=0$ for every $f\in C_c^{\infty}(\tLg(F))$ with the property that $J^{st}(\til{X},f)=0$ for all $\til{X}\in\tLg^{rs}(F)$.\p

For quasi-split unitary groups, Assem's conjecture asserts that all stable distributions supported on the nilpotent cone can be written into a linear combination of stable distributions where each term is a linear combination of nilpotent orbital integrals on various orbits in a single stable orbit. Moreover, on each stable nilpotent orbit there is a unique (up to constant) linear combination of the orbits for which the integral becomes stable.\p

Assem also had a conjecture regarding endoscopic transfer of nilpotent orbits. For endoscopic transfer of unitary groups, relevant elliptic endoscopy groups are products of two quasi-split unitary groups $U_{n_1}(E/F)\times U_{n_2}(E/F)$, while the target of endoscopy is $U_{n_1+n_2}(E/F)$. Here the three unitary groups split over the same quadratic extension.\p

Recall that if $\bH$ (e.g. $\bH=U_{n_1}(E/F)\times U_{n_2}(E/F)$) is an endoscopy group for $\tG$ with $\til{\mathfrak{h}}=\text{Lie }\bH$, then the transfer conjecture (for the Lie algebra) asserts that for any $f\in C_c^{\infty}(\tLg(F))$, there exists a function $f^{\bH}\in C_c^{\infty}(\til{\mathfrak{h}}(F))$ such that
\[\sum_{\til{Y}\sim\til{X}}\kappa(\til{Y})J(\til{Y},f)=J^{st}(\til{X},f^{\bH}),\;\forall\til{X}\in\til{\mathfrak{h}}^{\tG-rs}(F).\]

Here $\til{X}$ is $\tG$-regular \cite[2.2]{Wald} and $\til{Y}$ runs over regular semisimple orbits in $\tLg(F)$ that ``matches'' with $Y$. Also $\kappa=\kappa_{\bH}$ is some character (determined by $\bH$) on the set of orbits of such $\til{Y}$.\p

The transfer conjecture was proved by Waldspurger \cite{Wald2} conditional on Ng\^{o}'s later marvelous proof \cite{Ngo} on the fundamental lemma. Given the transfer conjecture, for any stable distribution $\mathcal{D}$ on $C_c^{\infty}(\mathfrak{h}(F))$, we can define its endoscopic transfer to be the distribution $\mathcal{D}^{\tG}:f\mapsto\mathcal{D}(f^{\bH})$. It's obvious that such distributions has to be $\tG(F)$-conjugation invariant.\p

If $\mathcal{D}$ is a stable distribution supported on the nilpotent cone of $\mathfrak{h}$, i.e. it's a linear combination of nilpotent orbital integral that becomes stable, then $\mathcal{D}^{\tG}$ has to be also supported on the nilpotent cone. It thus makes sense to talk about endoscopic transfer of nilpotent orbital integrals.\p

In \cite{Wald}, assuming $p$ large enough, Waldspurger completed the study of stability and endoscopic transfer (classical endoscopy, in the sense of Langlands-Shelstad \cite{LS}) for nilpotent orbital integrals for unramified classical groups. In particular, Assem's conjectures (see e.g. Conjecture \ref{Assem} and \ref{Assem2}) were proved in these cases.\p

The endoscopy data and the transfer factor, etc, are computed in \cite[Chap. X]{Wald}. These data as well as Waldspurger's formula can be equally stated when $E/F$ is ramified. The main goal of this section is to show that Theorem \ref{main}, \ref{tilded}, \ref{even} and \ref{tilded2} provide special cases and evidence on that Waldspurger's result could equally holds for ramified unitary groups, as well evidence for Assem's conjecture.\p

\subsection{Stability}\label{substab} We state Assem's stability conjecture in the unitary case.\p

\begin{conjecture}\cite[Conj. C, pp. 2]{Assem} \label{Assem}
Let $F$ be a non-archimedian local field with $\text{char}(F)=0$ or $\text{char}(F)\gg0$. For every stable nilpotent orbit $\CO$ of a quasi-split unitary group $U_n(E/F)$, there should be (up to constant) a unique linear combination of orbital integrals among the orbits in $\CO$ that gives a stable distribution. All stable distributions supported on the nilpotent cone can be written as a linear combination of such stable distributions.
\end{conjecture}\p

Waldspurger gave explicit formula for these combinations. If we restrict our attention to nilpotent orbits with (at most) two Jordan blocks, the formula of Waldspurger is simplified. For $0\le m\le \frac{n}{2}$, denote by $\CO_m(0)$ the set of nilpotent orbits with two Jordan blocks of sizes $n-m$ and $m$ (or regular nilpotent if $m=0$).\p

\begin{theorem} (Waldspurger, \cite[IX.15]{Wald})\label{Waldspurger1} Suppose $E/F$ is unramified, $\text{char}(k)>3n+1$, and $\text{char}(F)=0$. Then\hp

(a) For any $0\le m\le\frac{n}{2}$ with $2|mn$, 
\mhp

\[\sum_{\til{N}_m\in\CO_m(0)}J(\til{N}_m,\cdot)\text{ is a stable distribution.}\]\hp

(b) For any $0<m<\frac{n}{2}$ with $2\nmid mn$, we have $\#\CO_m(0)=2$. Denote by $\til{N}_m^{(1)}$ and $\til{N}_m^{(2)}$ these two orbits, then
\mhp

\[J(\til{N}_m^{(1)},\cdot)-J(\til{N}_m^{(2)},\cdot)\text{ is a stable distribution.}\]
\end{theorem}\p

What we can prove using Theorem \ref{main} and Theorem \ref{even} is\hp

\begin{theorem}\label{stable} Suppose instead $E/F$ is ramified. Under the assumption $\text{char}(k)\not=2$ and either $\text{char}(F)=0$ or $\text{char}(F)>n$, we have\hp

(i) The same results in Theorem \ref{Waldspurger1} are true for $m\le 2$.\qp

(ii) The same results in Theorem \ref{Waldspurger1} are true for all $m$ assuming Conjecture \ref{Assem} of Assem.
\end{theorem}\p

\begin{proof} We take $\ell$ an even integer so that $\til{N}$ is conjugate to $\pi^{\ell}\til{N}$ for every nilpotent $\til{N}\in\tLg(F)$ (see e.g. \cite[Sec. 3.1]{Ts}). Most of the time (e.g. when $\text{char}(F)=0$) $\ell=2$ works.\p

Let $\til{T}\in\tLg(F)_{\bx,-1/2}$ be any lift of $T\in\Lg(1)^{rs}(k)$ as in the introduction. The theorem of Shalika states that, for every function $f\in C_c^{\infty}(\tLg(F))$, there exists $N_0$ such that $\forall N\ge N_0$, if we write $f_{(N)}(X)=f(\pi^{\ell N}X)$, then 
\[J^{st}(\til{T},f_{(N)})=\sum_{\CO\in\CO(0)}\Gamma^{st}_{\CO}(\til{T})J(\CO,f_{(N)}).\]\p

Nilpotent orbital integrals have the property (due to the symplectic structure on $\CO$) that $J(\CO,f_{(N)})=q^{\frac{\ell N\dim\CO}{2}}J(\CO,f)$. This allows us to rewrite
\[J^{st}(\pi^{-\ell N}\til{T},f)=J^{st}(\til{T},f_{(N)})=\sum_{d=0}^{(\dim \tG-\text{rk}_{\bar{F}}\tG)/2}\sum_{\CO\in\CO(0),\dim\CO=2d}\Gamma^{st}_{\CO}(\til{T})q^{\ell dN}J(\CO,f).\]\hp

Now let $f$ be any ``stable'' function; $J^{st}(\til{X},f)=0$ for every $\til{X}\in\tLg^{rs}(F)$. The LHS by very definition vanishes. Interpolating with enough different $N$, we see that for every $d$,
\[\sum_{\CO\in\CO(0),\dim\CO=2d}\Gamma^{st}_{\CO}(\til{T})J(\CO,f)=0.\]

In other words
\[\sum_{\CO\in\CO(0),\dim\CO=2d}\Gamma^{st}_{\CO}(\til{T})J(\CO,\cdot)\text{ is a stable distribution.}\]\p

When $d=(\dim \tG-\text{rk}_{\bar{F}}\tG)/2-m$ with $m\le 2$, the only nilpotent orbits with dimension $2d$ are those nilpotent orbits with two Jordan blocks of sizes $n-m$ and $m$ (or one with size $n$ if $m=0$). To use previous results on Shalika germs, we need\p

\begin{lemma}\label{parity}For any $0\le m\le\frac{n}{2}$, there exists $T\in\Lg(1)^{rs}(k)$ such that $a_m(T)\not=0$.
\end{lemma}

\begin{proof}We use a parity trick. For each $0\le m<\frac{n}{2}$, we claim the existence of some $T$ for which $a_m(T)$ is odd. When $m=g+1$ and $n=2g+2$, we observe that $\frac{a_{g+1}(T)}{q+1}$ has the same parity as $\frac{a_g(T)}{q+1}$ and thus we reduce to the case $m=g$.\p

The idea is that the hyperelliptic involution gives an involution on $\Sym^m(C_T)(k)$. The parity of $\#\Sym^m(C_T)(k)$ is thus given by the number of fixed points that are defined over $k$, which in terms depends on the Galois structure on the Weierstrass points, or equivalently, how the characteristic polynomial factors in $k[x]$.\p

Using Lemma \ref{two} (ii),(iii) one can show the following: take $T$ so that $p_T(x)$ is an irreducible separable monic degree $n$ polynomial. Take $T'$ so that $p_{T'}(x)$ is another separable monic polynomial with two irreducible factors of degree $m$ and $n-m$. Then $a_m(T)\not\equiv a_m(T')$ (mod $2$).
\end{proof}\hp
%one checks the following: write $\hat{b}_m(T)$ the number of elements in $\Sym^m(C_T)(k)$ that are fixed by the hyperelliptic involution, whose corresponding degree $m$ divisor does not contain a hyperelliptic class. Then $\hat{b}_m(T)\equiv\hat{a}_m(T)$ (mod $2$). On the other hand, $\hat{b}_m(T)$ is simply the number of Galois-stable choices of $m$ Weierstrass points among the $2g+2$ ones.\p

Theorem \ref{main} and Theorem \ref{even} give us $\Gamma_{\CO}^{st}(\til{T})=a_m(T)$, or $-a_m(T)$ for one of the orbits if both $m$ and $n$ are odd. This completes part (i) of the theorem. For part (ii), simply note that Conjecture \ref{Assem} allows us to separate nilpotent orbits with two Jordan blocks out (or nilpotent orbits of any type of Jordan blocks) for stability question.\end{proof}\hp

\subsection{Endoscopic transfer}\label{subendo} The flow of this subsection is parallel to the previous section. However we will encounter interesting geometric and combinatorial identities that can be thought as consequences of endoscopy. Recall that our endoscopy group of $\tG=U_n(E/F)$ is $U_{n_1}(E/F)\times U_{n_2}(E/F)$ with $n_1+n_2=n$. We write $\tLg_1=\text{Lie }U_{n_1}(E/F)$ and $\tLg_2=\text{Lie }U_{n_2}(E/F)$.\p

We begin by stating the corresponding conjecture of Assem. The original conjecture of Assem for endoscopic transfer of nilpotent orbits comes from an induction construction due to Lusztig and makes use of the Springer correspondence (see \cite[4.3]{Assem}). In our case $U_{n_1}(E/F)\times U_{n_2}(E/F)$ is isomorphic to a twisted Levi subgroup of $U_n(E/F)$, and the construction agrees with that of Lusztig and Spaltenstein \cite{LuS}.\p

We summary their construction: Let $\mathbf{G}$ be a reductive group over an algebraically closed field $\bar{F}$ and $\mathbf{M}$ a Levi subgroup. Take $\mathbf{P}=\mathbf{M}\mathbf{N}\subset\mathbf{G}$ any parabolic subgroup for the Levi, where $\mathbf{N}$ is its unipotent radical. For any nilpotent orbit $\CO$ of $\text{Lie }\mathbf{M}$, the variety $\CO\cdot\text{Lie }\mathbf{N}$ has a dense open subset contained in some nilpotent orbit $\CO'$ of $\mathbf{G}$. We then denote $\text{ind}_{\mathbf{M}}^{\mathbf{G}}\CO:=\CO'$. In general when the reductive groups are defined over $F$, this should be understood as an induction between stable orbits.\p

In our case, $\mathbf{G}=U_n(E/F)$ and $\mathbf{M}=U_{n_1}(E/F)\times U_{n_2}(E/F)$ with $n_1+n_2=n$. The induction for nilpotent orbits with two Jordan blocks is especially clear: if $\til{N}_{m_1}^1$ and $\til{N}_{m_2}^2$ are stable nilpotent orbits in $\tLg_1$ (resp. $\tLg_2$) with two Jordan blocks of sizes $(n_1-m_1,m_1)$ and $(n_2-m_2,m_2)$ where $2m_i\le n_i$, then $\text{ind}_{U_{n_1}(E/F)\times U_{n_2}(E/F)}^{U_n(E/F)}\til{N}_{m_1}^1\times\til{N}_{m_2}^2=\til{N}_m$, the stable nilpotent orbits with two Jordan blocks of sizes $(n-m,m)$, with $m=m_1+m_2$.\p

We now assume Conjecture \ref{Assem}. Consider any stable combination $\mathcal{D}_1$ of nilpotent orbital integrals of $U_{n_1}(E/F)$ on orbits in a stable nilpotent orbit $\CO_1\subset\tLg_1(F)$, and likewise another stable combination $\mathcal{D}_2$ of $U_{n_2}(E/F)$ on orbits in a stable nilpotent orbit $\CO_2\subset \tLg_2(F)$. They give a stable nilpotent distribution $\mathcal{D}_1\otimes\mathcal{D}_2$ on $\tLg_1(F)\times\tLg_2(F)$ by $\mathcal{D}_{1}\otimes\mathcal{D}_{2}(f_1\otimes f_2)=\mathcal{D}_1(f_1)\mathcal{D}_2(f_2)$.\p

\begin{conjecture}\cite[Conj. D, pp. 83]{Assem}\label{Assem2} The endoscopic transfer of $\mathcal{D}_{1}\otimes\mathcal{D}_{2}$ given above is a linear combination of nilpotent orbital integrals on orbits which lie in the stable orbit $\text{ind}\;\!_{U_{n_1}(E/F)\times U_{n_2}(E/F)}^{U_n(E/F)}(\CO_1\times\CO_2)$.
\end{conjecture}\p

As in Theorem \ref{Waldspurger1}, Waldspurger proved the conjecture in the case of unramified classical groups. We formulate some cases of Waldspurger's result with two Jordan blocks. Let $\CO_{m_1}^1(0)$ be the set of nilpotent orbits of $U_{n_1}(E/F)$ with two Jordan blocks of sizes $n_1-m_1$ and $m_1$. Similarly $\CO_{m_2}^2(0)$ and $\CO_m(0)$ are used for nilpotent orbits of $U_{n_2}(E/F)$ and $U_n(E/F)$. We have\p

\begin{theorem}\label{Waldspurger2} (Waldspurger, \cite[XII.9]{Wald}) Suppose $E/F$ is unramified and $\text{char}(k)>3n+1$. Then\hp

(a) Suppose $n_1$ is odd and $n_2$ is even, so that $n=n_1+n_2$ is odd. Fix $0\le 2m_1<n_1$, $0\le 2m_2\le n_2$ and write $m=m_1+m_2$. Write $\epsilon\in F^{\times}/N_{E/F}E^{\times}$ for the non-trivial class. For nilpotent orbit $\til{N}_m\in\CO_m(0)$, put $\gamma(\til{N}_m)=-1$ if $m$ is odd and $\til{N}_m$ is the orbit classified by $((n-m,m),\epsilon\pi^{-1/2},(-1)^g)$. In all other cases put $\gamma(\til{N}_m)=1$. We define likewise the factor $\gamma(\til{N}_{m_1})$ for $\til{N}_{m_1}\in \CO_{m_1}^1(0)$. Then
\[\sum_{\til{N}_{m}\in\CO_{m}(0)}\gamma(\til{N}_m)^{m_1}J(\til{N}_m,\cdot)\]

is the endoscopic transfer of the stable distribution
\[\sum_{\til{N}_{m_1}\in\CO_{m_1}^1(0)}\gamma(\til{N}_{m_1})\sum_{\til{N}_{m_2}\in\CO_{m_2}^2(0)}J(\til{N}_{m_1},\cdot)\otimes J(\til{N}_{m_2},\cdot).\]\p

(b) Suppose both $n_1$ and $n_2$ are even, so that $n=n_1+n_2$ is also even. Fix $0\le 2m_1\le n_1$, $0\le 2m_2\le n_2$ and write $m=m_1+m_2$. For any $\til{N}_m\in\CO_m(0)$, put $\gamma(\til{N}_m)=1$ if $\til{N}_m$ is hyperbolic (see Definition \ref{ell}) and $\gamma(\til{N}_m)=-1$ if $\til{N}_m$ is elliptic. Then
\[\sum_{\til{N}_{m}\in\CO_{m}(0)}\gamma(\til{N}_m)^{m_1}J(\til{N}_m,\cdot)\]

is the endoscopic transfer of the stable distribution
\[\sum_{\til{N}_{m_1}\in\CO_{m_1}^1(0)}\sum_{\til{N}_{m_2}\in\CO_{m_2}^2(0)}J(\til{N}_{m_1},\cdot)\otimes J(\til{N}_{m_2},\cdot).\]
\end{theorem}\hp

Parallel to Theorem \ref{stable}, what we can show using Theorem \ref{main}, \ref{tilded}, \ref{even} and \ref{tilded2} is\p

\begin{theorem}\label{endo} Suppose instead $E/F$ is ramified. Then\hp

(i) The same results in Theorem \ref{Waldspurger2} are true for $m=m_1+m_2\le 2$.\qp

(ii) The same results in Theorem \ref{Waldspurger2} are true for all $m_1$, $m_2$ assuming Conjecture \ref{Assem} and Conjecture \ref{Assem2}.
\end{theorem}\hp

\begin{proof} We only give the proof for case (a). The proof for case (b) is completely the same while replacing the role of Theorem \ref{main} and \ref{tilded} by Theorem \ref{even} and \ref{tilded2}. The reason that in case (b) we want to assume both $n_1$ and $n_2$ are even (instead of only $n=n_1+n_2$ is even) is that in Theorem \ref{tilded2} we are only able to compute $\Gamma_{\til{N}_m}(\til{T})+\Gamma_{\til{N}_m}(\ad(u)\til{T})$. It happens that this discrepancy matters exactly when $n_1$ is odd.\p

The idea is similar to the proof of Theorem \ref{stable}. Let $\bx_1$ be a vertex on the Bruhat-Tits building of $U_{n_1}(E/F)$ with reductive quotient $\text{SO}_{n_1}/_k$. Let $V_1$ be the quasi-split quadratic space which is the standard representation of this $\text{SO}_{n_1}$. Let $T_1$ be any regular semisimple self-adjoint endomorphism of $V_1$. We have the same notations for $U_{n_2}(E/F)$ and let $T_2$ be any regular semisimple self-adjoint endomorphism of $V_2$.\p

Let $p_{T_1}(x),p_{T_2}(x)\in k[x]$ denote respectively the monic characteristic polynomials of $T_1$ and $T_2$. We assume that $p_{T_1}(x)$ and $p_{T_2}(x)$ are coprime. Write $C_{T_1}=(y^2=p_{T_1}(x))$, $C_{T_2}=(y^2=p_{T_2}(x))$, $J_{T_1}=\text{Pic}^0(C_{T_1})$ and $J_{T_2}=\text{Pic}^0(C_{T_2})$. By abuse of notation (as we don't have $T$ yet), we also write $p_T(x)=p_{T_1}(x)p_{T_2}(x)$ a degree $n$ monic polynomial, $C_T=(y^2=p_T(x))$ and $J_T=\text{Pic}^0(C_T)$.\p

The $G(0)(k)$-orbit of actual such $T\in\Lg(1)^{rs}(k)$ with characteristic polynomial $p_T(x)$ is a torsor under $H^1(k,J_T[2])$. This torsor is in fact canonically trivial \cite[Prop. 4]{BG} as mentioned in the introduction; the identity element in $H^1(k,J_T[2])$ corresponds to the $T$ for which $(\times2)^{-1}(\infty)\subset F_T$ is a trivial $J_T[2]$-torsor (Theorem \ref{Jerry}, \cite[Prop. 4]{BG} and \cite[Cor. 2.5 and Prop. 2.29]{Jerry}). This orbit of $T$ is also the one that intersects with the Kostant section \cite[Sec. 7]{BG2}. From now on we'll use the symbol $T$ to denote a representative of this orbit for which $(\times2)^{-1}(\infty)$ is trivial.\p

Let $\til{T}_1\in\tLg_1(F)_{\bx_1,-1/2}$ be a lift of $T_1$ and likewise for $\til{T}_2$. The orbits of those $\til{T}\in\tLg(F)$ that ``matches'' with $(\til{T}_1,\til{T}_2)\in\tLg_1(F)\times\tLg_2(F)$, i.e. that has the same characteristic polynomial, enjoy a one-one correspondence with those orbits of $T$ classified by $H^1(k,J_T[2])$ in the last paragraph, thanks to Lemma \ref{Galois}.\p

Recall that $J_T[2]\cong\text{Res}_k^{k[x]/p_T(x)}\mu_2/\mu_2\cong\ker(\text{Res}_k^{k[x]/p_T(x)}\mu_2\xra{Nm}\mu_2)$. In the middle group the $\mu_2$ is embedded into $\text{Res}_k^{k[x]/p_T(x)}\mu_2$ via the diagonal embedding. The second group and the third group are also dual to each other; this gives a self-dual structure $J_T[2]\times J_T[2]\ra\mu_2$.\p

Now as $p_T(x)=p_{T_1}(x)p_{T_2}(x)$, we have $\text{Res}_k^{k[x]/p_T(x)}\mu_2=\text{Res}_k^{k[x]/p_{T_1}(x)}\mu_2\times\text{Res}_k^{k[x]/p_{T_2}(x)}\mu_2$. On the latter group that is an element $\kappa=\kappa_{n_1,n_2}:=(1,-1)$. Since $\deg p_{T_2}=n_2$ is even, this element lies in $H^0(k,\ker(\text{Res}_k^{k[x]/p_T(x)}\mu_2\xra{Nm}\mu_2))\cong H^0(k,J_T[2])\cong H^1(k,J_T[2]^*)^*\cong H^1(k,J_T[2])^*$. In other words, $\kappa$ defines a character on $H^1(k,J_T[2])$.\p

By carefully checking the transfer factor, one can conclude that

\[\sum_{\alpha\in H^1(k,J_T[2])}\kappa(\alpha)J(\til{T}_{\alpha},\cdot)\]

is the endoscopy transfer of
\[J^{st}(\til{T}_1,\cdot)\otimes J^{st}(\til{T}_2,\cdot)=\sum_{\alpha_1\in H^1(k,J_{T_1}[2])}\sum_{\alpha_2\in H^1(k,J_{T_2}[2])}J(\til{T}_{\alpha_1},\cdot)\otimes J(\til{T}_{\alpha_2},\cdot).\]\p

Here $\til{T}_{\alpha}\in\tLg(F)$ is any representative of the orbit classified by $\alpha$ as described, and similarly for $\til{T}_{\alpha_1}\in\tLg_1(F)$, $\til{T}_{\alpha_2}\in\tLg_2(F)$. Arguing as in the proof of Theorem \ref{stable} and assume Conjecture \ref{Assem} and \ref{Assem2} if $m>2$, we have

\[\sum_{\til{N}_m\in\CO_m(0)}\left(\sum_{\alpha\in H^1(k,J_T[2])}\kappa(\alpha)\Gamma_{\til{N}_m}(\til{T}_{\alpha})\right)J(\til{N}_m,\cdot)\]\hp

is the endoscopy transfer of

\[\sum_{\tiny\begin{matrix}m_1+m_2=m\\0\le 2m_1<n_1\\0\le 2m_2\le n_2\end{matrix}}\sum_{\til{N}_{m_1}\in\CO_{m_1}^1(0)}\sum_{\til{N}_{m_2}\in\CO_{m_2}^2(0)}\Gamma^{st}_{\til{N}_{m_1}}(\til{T}_1)\Gamma^{st}_{\til{N}_{m_2}}(\til{T}_2)\cdot J(\til{N}_{m_1},\cdot)\otimes J(\til{N}_{m_2},\cdot).\]\p

Later we will simply write $m_1+m_2=m$ for the first summation in the last formula while it should be understood that $m_1$ and $m_2$ vary only in the range for which $\til{N}_{m_1}$ and $\til{N}_{m_2}$ are defined. The key is to prove\p

\begin{proposition}\label{propendo} We have equality
\[\sum_{\alpha\in H^1(k,J_T[2])}\kappa(\alpha)\Gamma_{\til{N}_m}(\til{T}_{\alpha})=\sum_{m_1+m_2=m}\gamma(\til{N}_{m_1})\Gamma^{st}_{\til{N}_{m_1}}(\til{T}_1)\Gamma^{st}_{\til{N}_{m_2}}(\til{T}_2),\]

where in the summation in the RHS, $\til{N}_{m_1}$ is chosen arbitrarily in $\CO_{m_1}^1(0)$ and $\til{N}_{m_2}$ is chosen arbitrarily in $\CO_{m_2}^2(0)$. See the definition of $\gamma(\cdot)$ in the statement of Theorem \ref{Waldspurger2}.
\end{proposition}\hp

\begin{proof}Using Theorem \ref{main}, \ref{tilded} and \ref{even}, what we have to prove is the following geometric identity that underlies this endoscopic transfer:

\[\sum_{m_1+m_2=m}a_{m_1}(T_1)a_{m_2}(T_2)=\frac{1}{\#J_T[2](k)}\;\cdot\]
\mhp
\begin{equation}\label{long}
\left(\sum_{\alpha\in H^1(k,J_T[2])}\kappa(\alpha)\left(\sum_{0\le2\ell\le m}\#\widetilde{\Sym}^{m-2\ell}(C_{T_{\alpha}})(k)\cdot(-q)^{\ell}\cdot C_{\ell}(-g+m-2\ell-1)\right.\right.
\end{equation}
\mhp
\[\left.\left.-(q+1)\sum_{0<2\ell+1\le m}\#\widetilde{\Sym}^{m-2\ell-1}(C_{T_{\alpha}})(k)\cdot(-q)^{\ell}\cdot C_{\ell}(-g+m-2\ell)\right)\right)\]
\hp

We have to explain the slight abuse of notation here. Different $\alpha\in H^1(k,J_T[2])$ gives us the same $C_{T_{\alpha}}=C_T$. However, the definition of the \'{e}tale $J_T[2]$-cover $\widetilde{\Sym}^m(C_{T_{\alpha}})$ of $\Sym^m(C_T)$ depends on the orbit of $T_{\alpha}$, thus depends on $\alpha$. In fact, changing $\alpha\in H^1(k,J_T[2])$ exactly amounts to changing the Frobenius structure on $\widetilde{\Sym}^m(C_{T_{\alpha}})$ as a $J_T[2]$-torsor over $\Sym^m(C_T)$.\p

Recall that $T$ is used to denote the $T_{\alpha}$ with $\alpha$ trivial. For any $\kappa'\in H^1(k,J_T[2])^*=H^0(k,J_T[2]^*)$, we can consider the $\kappa'$-isotypic component $H^*(\widetilde{\Sym}^m(C_{T_{\alpha}}))_{\kappa'}$. We have
\[\text{Tr}(\text{Frob}:H^*(\widetilde{\Sym}^m(C_{T_{\alpha}})/_{\bar{k}},\Ql)_{\kappa'})=\kappa'(\alpha)\cdot\text{Tr}(\text{Frob}:H^*(\widetilde{\Sym}^m(C_T)/_{\bar{k}},\Ql)_{\kappa'})\]\hp

Summing over all $\alpha$ and all $\kappa'$, we see
\[\sum_{\alpha\in H^1(k,J_T[2])}\kappa(\alpha)\#\widetilde{\Sym}^m(C_{T_{\alpha}})(k)\]
\mhp
\begin{equation}\label{cover}
=(-1)^m\sum_{\alpha\in H^1(k,J_T[2])}\kappa(\alpha)\sum_{\kappa'\in H^1(k,J_T[2])*}\kappa'(\alpha)\cdot\text{Tr}(\text{Frob}:H^*(\widetilde{\Sym}^m(C_T)/_{\bar{k}},\Ql)_{\kappa'}).
\end{equation}

\[=(-1)^m\#J_T[2](k)\cdot\text{Tr}(\text{Frob}:H^*(\widetilde{\Sym}^m(C_T))_{\kappa}/_{\bar{k}},\Ql).\;\;\;\;\;\;\;\;\;\;\;\;\;\;\;\;\;\;\;\;\;\;\;\;\;\;\;\;\;\;\;\;\;\,\]\p

In the last step we used the equality $\#H^1(k,J_T[2])=\#J_T[2](k)$. To compute $\text{Tr}(\text{Frob}:H^*(\widetilde{\Sym}^m(C_T)/_{\bar{k}},\Ql)_{\kappa})$, it will be a good idea to first deal with the case $m=1$. In the rest of the proof we write $\til{C}_T:=\widetilde{\Sym}^1(C_T)$. This is an \'{e}tale $J_T[2]$-cover of $C_T$.\p

Finite covers between (projective smooth) curves can be read out from their function fields. Let's base change from the ground field $k$ to $\bar{k}$ for the moment. Recall $C_T$ is a double cover of $\mathbb{P}^1$. Their function fields are respectively $\bar{k}(x)\subset \bar{k}(x,\sqrt{p_T(x)})$. The key is to observe
\[\bar{k}(\til{C}_T)=\bar{k}(x,\sqrt{p_T(x)},\sqrt{P(x)}\,|\,P(x)\text{ runs over even degree divisors of }p_T(x)).\]\hp

This is because the above function field extension gives an \'{e}tale $J_T[2]$-cover of $C_T$, which is unique over $\bar{k}$. Now $\kappa$, being a non-trivial element in $J_T[2](k)^*$, corresponds to a degree $2$ cover $C_T^{\kappa}$ of $C_T$ inside $\til{C}_T\ra C_T$. This cover is given by the function field $\bar{k}(C_T^{\kappa})=\bar{k}(x,\sqrt{p_T(x)},\sqrt{p_{T_2}(x)})=\bar{k}(x,\sqrt{p_{T_1}(x)},\sqrt{p_{T_2}(x)})$.\p

The curve $C_T^{\kappa}$, as well as its function field, descend back to $k$. Precisely, since $\til{C}_T$ is defined to be the curve for which the fiber above $\infty$ is trivial, we have $C_T^{\kappa}=k(x,\sqrt{p_T(x)},\sqrt{p_{T_2}(x)})$ (here it's important that $p_{T_2}(x)$ was chosen to be monic). Now recall\p

\begin{lemma}\label{Koji}Let $X$ be a quasi-projective variety over any field $k$ and $G$ be a finite group acting on $X$. Choose prime $\ell$ which is coprime to $|G|$. Let $Y=X/G$ be the scheme-theoretic quotient. Then $H^*(Y/_{\bar{k}},\Ql)\cong H^*(X/_{\bar{k}},\Ql)^G$.
\end{lemma}\p

Using the lemma, we have
\[H^*(\til{C}_T/_{\bar{k}},\Ql)_{\kappa}=H^*(C_T^{\kappa}/_{\bar{k}},\Ql)\ominus H^*(C_T/_{\bar{k}},\Ql),\]\qp

where the two sides of the equality are in the abelian category of virtual representations of the free abelian group generated by Frobenius. Nevertheless, it's obvious from the function field of $C_T^{\kappa}$ that it is a $(\mu_2)^2$-cover of $\mathbb{P}^1$, and that the three double covers in the middle are $C_T$, $C_{T_1}$ and $C_{T_2}$! This gives
\[H^*(C_T^{\kappa}/_{\bar{k}},\Ql)\ominus H^*(C_T/_{\bar{k}},\Ql)\]
\[=\left(H^*(C_{T_1}/_{\bar{k}},\Ql)\ominus H^*(\mathbb{P}^1/_{\bar{k}},\Ql)\right)\oplus\left(H^*(C_{T_2}/_{\bar{k}},\Ql)\ominus H^*(\mathbb{P}^1/_{\bar{k}},\Ql)\right)\]
\[=H^1(C_{T_1}/_{\bar{k}},\Ql)\oplus H^1(C_{T_2}/_{\bar{k}},\Ql).
\;\;\;\;\;\;\;\;\;\;\;\;\;\;\;\;\;\;\;\;\;\;\;\;\;\;\;\;\;\;\;\;\;\;\;\;\;\;\;\;\;\;\;\;\;\;\;\;\;\;\;\;\;\;\;\,\]\p

In summary $H^*(\til{C}_T/_{\bar{k}},\Ql)_{\kappa}=H^1(C_{T_1}/_{\bar{k}},\Ql)\oplus H^1(C_{T_2}/_{\bar{k}},\Ql)$. For general $m$, what we have is\hp

\begin{lemma}\label{endogeoid} Let $J_{T_1}$ and $J_{T_2}$ be the Jacobian of $C_{T_1}$ and $C_{T_2}$, respectively. Then
\begin{equation}\label{unitary}
H^*(\widetilde{\Sym}^m(C_T)/_{\bar{k}},\Ql)_{\kappa}=\bigoplus_{d=0}^m H^d(J_{T_1}/_{\bar{k}},\Ql)\otimes H^{m-d}(J_{T_2}/_{\bar{k}},\Ql).
\end{equation}
\end{lemma}\hp

To prove Lemma \ref{endogeoid}, note that $\pi_1^{\acute{e}t,tame}(\Sym^m(C_T))\cong\pi_1^{\acute{e}t,tame}(C_T)$ canonically, and thus we have a double cover $\Sym^m(C_T)^{\kappa}\ra\Sym^m(C_T)$ corresponding to $C_T^{\kappa}\ra C_T$. This double cover can be seen as a $S_m$-quotient of $((C_T)^m)^{\kappa}\ra (C_T)^m$, the ``diagonal'' double cover in the $(\mu_2)^m$-cover $(C_T^{\kappa})^m\ra (C_T)^m$.\p

Now the cover $((C_T)^m)^{\kappa}$ is a $\left((\mu_2)^m\times\mu_2\right)$-cover of $(\mathbb{P}^1)^m$. For any $\nu\in\{1,2\}^m$, denote by $V^{\nu}$ the ``diagonal'' double cover of $(\mathbb{P}^1)^m$ in $\prod_{i=1}^mC_{T_{\nu(i)}}\ra(\mathbb{P}^1)^m$. These are exactly all the double covers of $(\mathbb{P}^1)^m$ which are between $((C_T)^m)^{\kappa}\ra(\mathbb{P}^1)^m$ but not between $(C_T)^m\ra(\mathbb{P}^1)^m$. We thus have 

\[H^*(((C_T)^m)^{\kappa}/_{\bar{k}},\Ql)\ominus H^*((C_T)^m/_{\bar{k}},\Ql)=\sum_{\nu\in\{1,2\}^m}\left(H^*(V^{\nu}/_{\bar{k}},\Ql)\ominus H^*((\mathbb{P}^1)^m/_{\bar{k}},\Ql)\right).\]\p

On the other hand, for the $(\mu_2)^m$-cover $\prod_{i=1}^mC_{T_{\nu(i)}}\ra(\mathbb{P}^1)^m$, we can consider the product map $\phi:\mu_2^m\ra\mu_2$. Then the $\phi$-isotypic part is equal to the term in the previous sum:
\[H^*(\prod_{i=1}^mC_{T_{\nu(i)}}/_{\bar{k}},\Ql)_{\phi}\cong H^*(V^{\nu}/_{\bar{k}},\Ql)\ominus H^*((\mathbb{P}^1)^m/_{\bar{k}},\Ql),\;\forall \nu\in\{1,2\}^m.\]\p

Nevertheless, from the K\"{u}nneth formula one deduces $H^*(\prod_{i=1}^mC_{T_{\nu(i)}}/_{\bar{k}},\Ql)_{\phi}=\bigotimes_{i=1}^m H^1(C_{T_{\nu(i)}}/_{\bar{k}},\Ql)$. Putting together, we have

\[H^*(((C_T)^m)^{\kappa}/_{\bar{k}},\Ql)\ominus H^*((C_T)^m/_{\bar{k}},\Ql)=\sum_{\nu\in\{1,2\}^m}\bigotimes_{i=1}^m H^1(C_{T_{\nu(i)}}/_{\bar{k}},\Ql).\]\p

Now the LHS of (\ref{unitary}) is the $S_m$-invariant part of the LHS above, taking $S_m$-invariant on the RHS gives
\[H^*(\widetilde{\Sym}^m(C_T)/_{\bar{k}},\Ql)_{\kappa}=\sum_{d=0}^m\Sym^d H^1(C_{T_1}/_{\bar{k}},\Ql)\otimes\Sym^{m-d}H^1(C_{T_2}/_{\bar{k}},\Ql),\]\hp

where on the RHS, $d$ correspond to the number of $i$ with $\nu(i)=1$. Here (!) the $\Sym^dH^1$ above has the meaning of the $d$-th symmetric power of (virtual) representations as super (i.e. ($\Z/2\Z$)-graded) vector spaces; that is, $\Sym^dH^1=\bigwedge^dH^1$ in the usual notation. This proves (\ref{unitary}).\p

Combining (\ref{cover}) and (\ref{unitary}), we obtain

\begin{equation}\label{combine}
\frac{1}{\#J_T[2](k)}\sum_{\alpha\in H^1(k,J_T[2])}\kappa(\alpha)\#\widetilde{\Sym}^m(C_{T_{\alpha}})(k)=\bigoplus_{d=0}^m H^d(J_{T_1}/_{\bar{k}},\Ql)\otimes H^{m-d}(J_{T_2}/_{\bar{k}},\Ql).
\end{equation}\p

It is now a matter of combinatorics to prove (\ref{long}). First we have to rewrite $a_{m_1}(T)$ and $a_{m_2}(T)$. In the odd case, that is for $a_{m_1}(T)$, Lemma \ref{two}(iii) and Proposition \ref{inverse} together gives

\[a_{m_1}(T)=\sum_{0\le2\ell\le m_1}q^{\ell}\cdot C_{\ell}(-g_1+m_1-2\ell)\cdot \text{Tr}(\text{Frob}:H^{m_1-2\ell}(J_{T_1}/_{\bar{k}},\Ql)),\]\hp

where $g_1$ is the genus of $C_{T_1}$; $n_1=2g_1+1$. For the even case, the number $a_m(T)$ is like $a_m(T)-(q+1)a_{m-1}(T)$ if using the definition of the odd case. This gives

\[a_{m_2}(T)=\sum_{0\le2\ell\le m_2}q^{\ell}\cdot C_{\ell}(-g_1+m_2-2\ell)\cdot\text{Tr}(\text{Frob}:H^{m_2-2\ell}(J_{T_2}/_{\bar{k}},\Ql))\]
\mhp
\[-(q+1)\sum_{0<2\ell+1\le m_2}q^{\ell}\cdot C_{\ell}(-g_1+m_2-2\ell-1)\cdot\text{Tr}(\text{Frob}:H^{m_2-2\ell-1}(J_{T_2}/_{\bar{k}},\Ql)).\]\p

Having the expressions of $a_{m_1}(T)$ and $a_{m_2}(T)$ at hand, one sees that (\ref{long}) follows from (\ref{combine}) and Corollary \ref{multiCata}. This finishes the proof of Proposition \ref{propendo}.\end{proof}\p

With the endoscopic transfer formula we had right before Proposition \ref{propendo}, it now suffices to show that when we run over all possible choices of coprime separable polynomials $p_{T_1}(x),p_{T_2}(x)\in k[x]$ of degree $n_1$ and $n_2$, respectively, we have
\[\sum_{m_1+m_2=m}\sum_{\til{N}_{m_1}\in\CO_{m_1}^1(0)}\sum_{\til{N}_{m_2}\in\CO_{m_2}^2(0)}\gamma(\til{N}_{m_1})a_{m_1}(T_1)a_{m_2}(T_2)\cdot J(\til{N}_{m_1},\cdot)\otimes J(\til{N}_{m_2},\cdot)\]

spans the linear space of stable distributions supported on the union of all $\til{N}_{m_1}\times\til{N}_{m_2}$ with $m_1+m_2=m$.\p

In other words we have to prove the vectors $(a_{m_1}(T_1)a_{m_2}(T_2))_{m_1+m_2=m}$ for different $T_1,T_2$ span $\Q^{\{(m_1,m_2)\,|\,m_1+m_2=m,\;0\le 2m_1<n_1,\;0\le 2m_2\le n_2\}}$. That this is always the case can be proved with a parity trick similar to Lemma \ref{parity}. This finishes the proof of Theorem \ref{endo}.\end{proof}\p

\begin{remark}\label{rmkconj} In fact, it was endoscopic transfer which led us into conjecturing the results in Theorem \ref{main} and Theorem \ref{even} before knowing how to compute them. The point is that without having a good method to compute Shalika germs, Section \ref{secgeom} already tells us that the stable Shalika germs $\Gamma_{\til{N}_m}(\til{T})$ should be expressed in terms of linear combinations of $\#\Sym^{m'}(C_T)(k)$, $0\le m'\le m$. Together with Assem's conjectures, this suggests that something like (\ref{long}), with some a priori unknown coefficients, should be true.\p

On the other hand, $\til{N}_m$ only exists as a nilpotent orbit with codimension $2m$ in the regular nilpotent orbit if $2m\le n$. In other words, this suggests that the stable Shalika germ formula (which we proved to be $a_m(T)$), should be something that vanishes when $2m>n$. This together with some weaker computation was what led us to the formula $\Gamma_{\til{N}_m}^{st}(\til{T})=a_m(T)$.
\end{remark}\p

\hp
\section{Local character expansions of supercuspidal representations}\label{charsc}

This section is devoted to the application of our Shalika germ formulas to local character expansion of specific supercuspidal representations. Briefly speaking, we use our result on supercuspidal representations whose local character looks like the Fourier transform of $J(\til{T},\cdot)$ to obtain a Harish-Chandra-Howe local character expansion, and invoke the interpretation of character expansion by M{\oe}glin-Waldspurger \cite{MW}. For the use of the result in \cite{MW}, we assume $\text{char}(F)=0$ in this section.\p

Again fix $T\in\Lg(1)^{rs}(k)$. Recall $\Lg(1)(k)\cong\tLg(F)_{\bx,-1/2}/\tLg(F)_{\bx,0}\cong\tLg(F)_{\bx,1/2}/\tLg(F)_{\bx,1}$. Furthermore we have that $\Lg(1)(k)$ is self-dual, allowing us to identify $\tLg(F)_{\bx,-1/2}/\tLg(F)_{\bx,0}$ and $\tLg(F)_{\bx,1/2}/\tLg(F)_{\bx,1}$ as the dual of each other. With a choice of non-trivial additive character $\psi:(k,+)\ra\mathbb{C}^{\times}$, the elment $T$ then give rises to a character on $\tLg(F)_{\bx,1/2}/\tLg(F)_{\bx,1}\cong\tG(F)_{\bx,1/2}/\tG(F)_{\bx,1}$, and thus a $1$-dimensional representation of $\tG(F)_{\bx,1/2}$. We denote by $\psi_T$ this representation.\p

The compact induction

\mhp
\[\pi_T:=\text{c-ind}_{\tG(F)_{\bx,1/2}}^{\tG(F)}\psi_T=\{f\in C_c^{\infty}(\tG(F))\;|\;f(g_1g_2)=\psi_T(g_1)f(g_2),\;\forall g_1\in\tG(F)_{\bx,1/2}\}.\]\qp

can be shown to be the direct sum of finitely many supercuspidal representations. Let $r$ be the number of irreducible factor of the characteristic polynomial $p_T(x)$ of $T$ and $L=k[x]/p_T(x)$ be an \'{e}tale algebra over $k$; $L$ is the direct product of $r$ finite extensions of $k$. We have $\Stab_{O(V)}(T)=\text{Res}_k^L\mu_2$ has $2^r$ points defined over $k$. Then $\pi_T$ is the direct sum of $2^r$ distinct irreducible supercuspidal representations \cite[Prop. 2.4]{RY} of depth $\frac{1}{2}$. These are examples of {\it epipelagic representations} of Reeder and Yu \cite{RY}.\p

Now let $\Phi_{\pi_T}$ be the character of $\pi_T$. In other words, $\Phi_{\pi_T}\in C_c^{\infty}(\tLg(F))$ is the ($\tG(F)$-conjugation) invariant distribution such that for any $f\in C_c^{\infty}(\tG(F))$, $\Phi_{\pi_T}(f):=\text{Tr}(\pi_T(f))$. Here to define $\pi_T(f)$ we need a choice of measure on $\tG(F)$, which we give in Appendix \ref{norm}.\p

The basic philosophy that goes back to at least Harish-Chandra is that characters should be compared with Fourier transforms of orbital integrals. Use as in Appendix \ref{norm} the self-dual structure $\psi(B(\cdot,\cdot))$ and measure on $\tLg(F)$. This gives, for $f\in C_c^{\infty}(\tLg(F))$, its Fourier transform
\[\widehat{f}(\til{X}):=\int_{\tLg(F)}\psi(B(\til{X},\til{Y}))d\til{Y}.\]\p

We define $\widehat{J}(\til{X},f):=J(\til{X},\widehat{f})$, the Fourier transform of orbital integrals. Fix a lift $\til{T}\in\tLg(F)_{\bx,1/2}$. What one has is that\p

\begin{lemma}\label{scchar}(i) $\Theta_{\pi_T}$ is supported on $\tG(F)_{\bx,1/2}$.\hp

(ii) Let $\se:\tLg(F)_{\bx,1/2}\xra{\sim}\tG(F)_{\bx,1/2}$ be a mock exponential map (see \cite[Hyp. 3.2.1]{De}, for us it can be given by the Cayley transform). Then for any $f\in C_c^{\infty}(\tLg(F)_{\bx,1/2})$,
\[\Theta_{\pi_T}(f\circ\se)=2^r\cdot\widehat{J}(\til{T},f).\]\hp

(iii) For each of the $2^r$ components of $\pi_T$, its character (which has larger support), when restricted to $\tG(F)_{\bx,1/2}$ and pulled back to $\tLg(F)_{\bx,1/2}$ via $\se$, is equal to $\widehat{J}(\til{T},f)$.
\end{lemma}\p

From now on let $\pi_T^o$ be any fixed component of $\pi_T$, and $\Phi_{\pi_T^o}$ its character. Let $\CO(0)$ be the set of nilpotent orbits. Then the Harish-Chandra-Howe local character expansion \cite[Thm. 4]{HC} states that there exists constants $(c_{\CO}(\pi_T^o))_{\CO\in\CO(0)}\in\mathbb{C}$ such that
\begin{equation}\label{HC}
\Phi_{\pi_T^o}(f\circ\se)=\sum_{\CO\in\CO(0)}c_{\CO}(\pi_T^o)\widehat{J}(\CO,f),
\end{equation}

for all $f$ that are supported in a sufficiently small neighborhood $U\subset\tLg(F)_{\bx,1/2}$ of $0\in\tLg(F)$. On the other hand, in \cite{MW} M{\oe}glin and Waldspurger proved that, if $\CO'$ is any nilpotent orbit satisfying that for any $\CO$ whose boundary contains $\CO'$ we have $c_{\CO}=0$, then $c_{\CO'}(\pi_T^o)$ is equal to the dimension of the degenerated Whittaker model associated to $\CO'$ of $\pi_T^o$.\p

Now, restricting to the small neighborhood $U$, we have by Lemma \ref{scchar}(iii) and (\ref{HC}) that
\[\widehat{J}(\til{T},f)=\sum_{\CO\in\CO(0)}c_{\CO}(\pi_T^o)\widehat{J}(\CO,f).\]

By inversing the Fourier transform, we see\p
 
\begin{corollary} We have $c_{\CO}(\pi_T^o)=\Gamma_{\CO}(\til{T})$, the latter are given by formulas in Theorem \ref{tilded} and \ref{tilded2}.
\end{corollary}\hp

\begin{corollary} For any $m\ge 0$ we can find $C=C(m,q)$ such that for any $n\ge C$, we can find supercuspidal representations of $U_n(E/F)$ of the form $\pi_T^o$ such that $c_{\CO}(\pi_T^o)=0$ for every nilpotent orbit $\CO$ of two Jordan blocks of sizes $n-m'$ and $m'$, $0\le m'\le m$. Here by abuse of language the $m'=0$ case corresponds to an orbit with a single Jordan block, namely a regular nilpotent orbits.
\end{corollary}

\begin{proof} Take $T$ so that $p_T(x)$ has as many irreducible factors as possible, so that $\#J_T[2](k)=2^r$ with $r>\frac{n}{1+\log_qn}$. The varieties in Theorem \ref{tilded} and \ref{tilded2} are $J_T[2]$-covers of $\Sym^{m'}(C_T)$ (and $\Sym^{m'}(C_T')$, etc), whose numbers of points can be bounded by the Weil bound on the Frobenius trace. Now for any $J_T[2]$-cover of $\Sym^{m'}(C_T)$, the fiber above a rational point in $\Sym^{m'}(C_T)$ is a $J_T[2]$-torsor. Recall that the orbits in the stable orbit of $T$ are classified by $H^1(k,J_T[2])$, and when $T$ runs over all such orbits in the same stable orbit, the fiber above any chosen rational point in $\Sym^{m'}(C_T)$ will also run over all possible $J_T[2]$-torsors.\p

Once $n$ is large enough, we have $r$ large enough so that $2^r$ will be much greater than $\sum_{m'=0}^m\#\Sym^{m'}(C_T)(k)$ (and more for other covers and covers of $\Sym^{m'}(C_T')$, etc). We can thus find an orbit in the stable orbit of $T$, i.e. a class in $H^1(k,J_T[2])$, such that for the corresponding covers $\widetilde{\Sym}^{m'}(C_T)$ in Theorem \ref{tilded} and \ref{tilded2}, $0\le m'\le m$, the torsor above each rational point is non-trivial. That is to say $\widetilde{\Sym}^{m'}(C_T)$ (and similarly $\widetilde{\Sym}^{m'}(C_T'), \widetilde{\Sym}^{m',*}(C_T)$, etc) has no rational points. When $n$ is odd this says $c_{\CO}(\pi_T^o)=\Gamma_{\CO}(\til{T})=0$, which is what we want. When $n$ is even we have instead $c_{\CO}(\pi_T^o)+c_{\Ad(u)\CO}(\pi_T^o)=\Gamma_{\CO}(\til{T})+\Gamma_{\Ad(u)\CO}(\til{T})=0$. Since inductively by \cite{MW} we have $c_{\CO}(\pi_T^o),c_{\Ad(u)\CO}(\pi_T^o)\ge 0$, we conclude that they all vanish.
\end{proof}\hp

Note a nilpotent with two Jordan blocks is never in the closure of a nilpotent orbit with more than two Jordan blocks. One can thus have many examples where the dimension of the degenerate Whittaker models are (up to constant) number of rational points on varieties in Theorem \ref{tilded} and \ref{tilded2}.\p

\begin{example}
For example, take $n=2g+1$ odd and take $p_T(x)\in k[x]$ any polynomial of degree $2g+1$ that is the product of $r$ distinct irreducible factors with $r>1$. Let $C_T=(y^2=p_T(x))$ (the smooth completion). Take an \'{e}tale Galois $J_T[2]$-cover $\til{C}_T$ of $C_T$ for which the fiber above $\infty\in C_T$ is a non-trivial $J_T[2]$-torsor. Such a choice corresponds to an orbit of such $T$ in its stable orbit. The corresponding representation has $c_{\til{N}_0}(\pi_T^o)=\frac{1}{\#J_T[2](k)}\#\widetilde{\Sym}^0(C_T)(k)=0$ and $c_{\til{N}_1}(\pi_T^o)=\frac{1}{\#J_T[2](k)}\#\til{C}_T(k)$, i.e. the dimension of the degenerate Whittaker model for the subregular orbit $\til{N}_1$ is $2^{-(r-1)}$ times the number of rational points on $\til{C}_T$, a curve of genus $2^{2g}(g-1)+1$ over $k$. It will be interesting to see how these points actually ``live'' on the degenerate Whittaker model.
\end{example}

\hp\appendix
\hp
\section{Normalization of measures}\label{norm}

This appendix is for the normalization of semisimple and nilpotent orbital integrals on our $p$-adic group $\tG$. Our normalization essentially follows that of \cite{MW}.\p

For $\til{X}\in\tLg(F)$ regular semisimple, our $J(\til{X},\cdot)$ is what is usually written $|D(\til{X})|^{1/2}\mu_{\til{X}}(\cdot)$. More precisely, let $D(\til{X}):=\det(\ad(\til{X})|_{\tLg/\tLg_{\til{X}}})$, where $\tLg_{\til{X}}$ denotes the centralizer of $\til{X}$. The norm $|\cdot|$ on $F$ is such that $|\pi|=q^{-1}$. We define
\[J(\til{X},f):=|D(\til{X})|^{1/2}\int_{\tG(F)/\tG_{\til{X}}(F)}f(\Ad(g)\til{X}).\]\hp

And the normalization of measures goes as follows. Fix an additive character $\psi:F\ra\mathbb{C}^{\times}$ such that $\psi$ is trivial on $\pi_F$ but not on $\CO_F$. Let $B(\cdot,\cdot):\tLg\times\tLg\ra\mathbb{G}_a$ be an $F$-Killing form on $\tLg$. In fact in the article we'll take $B(\cdot,\cdot)$ to be the naive trace form on the space of anti-hermitian spaces, which has the property that for any point $x'$ on the building and $d\in\mathbb{R}$, $\psi(B(\cdot,\cdot))$ identifies $\tLg(F)_{\bx',d:d+}$ as the dual of $\tLg(F)_{\bx',-d:(-d)+}$.\p

The Haar measure on $\tLg(F)$ is taken to be the one that is self-dual by $\psi\circ B$, and the Haar measure on $\tG(F)$ to be the one so that the (mock) exponential map is measure preserving near the identity. $\tLg_{\til{X}}\subset\tLg$ is a subspace on which $B(\cdot,\cdot)$ is non-degenerate, and the Haar measure on $\tLg_{\til{X}}(F)$ and $\tG_{\til{X}}(F)$ is defined in the same way by restricting $B(\cdot,\cdot)$ to $\tLg_{\til{X}}\times\tLg_{\til{X}}$. This defines the required Haar measure in the above regular semisimple orbital integral.\p

Lastly, the normalization of nilpotent orbital integrals goes as follows. We assume in this article that $\text{char}(F)=0$ or $\text{char}(F)>n$. This implies that any nilpotent orbit $\CO\subset\tLg$ is smooth with expected tangent space; for $N\in\CO$, we have $T_N{\CO}\cong\tLg/\tLg_N$. Now $\tLg/\tLg_N$ has a symplectic structure $B_N:(\til{X},\til{Y})\mapsto B([\til{X},\til{Y}],N)$.\p

We take the measure on $\CO$ to be given by the top wedge power of this symplectic form. More precisely, this measure has the following interpretation. Take a Lagrangian $F$-subspace $L\subset\tLg(F)/\tLg_N(F)$ and $\Lambda_L\subset L$ any lattice. Let $L'$ be any $F$-complement of $L$ and $\Lambda_L'=\{\til{X}\in L'\,|\,\psi(B_N(\til{X},\til{Y}))=1,\;\forall \til{Y}\in\Lambda_L\}$ be the dual lattice. Then $\Lambda_L+\Lambda_L'$ is assigned to have measure $1$.\p

\hp
\section{Catalan numbers}

This appendix discusses combinatorics that appear in analyzing Shalika germs and their endoscopic transfer consequence. We omit the proofs, which are fairly elementary.\p

\begin{definition}\label{Cata}For any integer $\ell\ge 0$, we define degree $\ell$ polynomials $C_{\ell}(x)\in\Q[x]$ by
\[C_{\ell}(x)=\frac{x}{(x+2\ell)\cdot\ell!}\prod_{i=1}^{\ell}(x+\ell+i).\]
\end{definition}\p

\begin{remark}$C_{\ell}(0)=0$ except for $C_{0}(x)\equiv 1$. Also $C_{\ell}(1)$ is the classical Catalan numbers $1,1,2,5,14,...$. See e.g. Wikipedia.
\end{remark}\p

\begin{proposition}\label{recur} For any integer $\ell>0$, $C_{\ell}(x+1)-C_{\ell}(x)=C_{\ell-1}(x+2)$.
\end{proposition}\p

The following observation was shown to me by Joel B. Lewis.\hp

\begin{proposition}Let $C(x,q):=\sum_{\ell=0}^{\infty}C_{\ell}(x)q^{\ell}$, we have
\[C(x,q)=\left(\frac{1-\sqrt{1-4q}}{2q}\right)^x.\]
\end{proposition}\p

\begin{corollary}\label{multiCata}
We have $C(x+y,q)=C(x,q)C(y,q)$. Equivalently $C_{\ell}(x+y)=C_{\ell}(x)C_0(y)+C_{\ell-1}(x)C_1(y)+...+C_0(x)C_{\ell}(y).$
\end{corollary}\p

\begin{proposition}\label{inverse}Let $A=(A_{ij})_{i,j\in\Z_{\ge 0}}$ be the lower triangular matrix with entries in $\Q[x,q]$ with
\[A_{ij}=\left\{\begin{matrix}
q^{\ell}\binom{x-j}{\ell}&\text{if }i=j+2\ell,\;\ell\in\Z_{\ge0}.\\
\\
0&\text{otherwise}.
\end{matrix}\right.\]\hp

Then the inverse of $A$ is given by
\[(A^{-1})_{ij}=\left\{\begin{matrix}
q^{\ell}C_{\ell}(-x+j)&\text{if }i=j+2\ell,\;\ell\in\Z_{\ge0}.\\
\\
0&\text{otherwise}.
\end{matrix}\right.\]
\end{proposition}\p

We add another vaguely related proposition, which is used in the end of subsection \ref{secnil}.\p

\begin{proposition}\label{path}Let $0\le m'\le m\le g$. Let $\Xi_{m,m'}\subset S_m$ be the subset of bijections of $\{1,...,m\}$ that satisfies an equivalent of Condition \ref{cond}:
$\sigma(i)=i$ for $i=1,...,m'$, and if either $j\le m$ and $j-i=2$, or $j>m$ and $j-i=1$, then we have $\sigma(j)>\sigma(i)$.\hp

Write $\delta_3(\sigma)=\#\{1\le i<g\,|\,\sigma(i)>\sigma(i+1)\}$. Then $\delta_3(\sigma)\le\floor{\frac{m-m'}{2}}$ for $\sigma\in\Xi_{m,m'}$ and for $0\le r\le\lfloor\frac{m-m'}{2}\rfloor$,
\[\#\{\sigma\in\Xi_{m,m'}\,|\,\delta_3(\sigma)\le r\}=\binom{g-m'}{r}.\]
\end{proposition}\p

\hp
\section{Restrictions on characteristic of local and residue fields}\label{char}

In this appendix we explain what restrictions are necessary, and why some others can be relaxed. Recall $F$ is the local field and $k$ its residue field. The restriction we have for the results in this paper is $\text{char}(k)\not=2$ and either $\text{char}(F)=0$ or $\text{char}(F)>n$, where $\tG=U_n(E/F)$ (except for subsection \ref{subendo} and Section \ref{charsc}, in which we furthermore require $\text{char}(F)=0$). Our main reference here is \cite[Appendix A]{Ts}. To begin with, if $\text{char}(F)\not=0$, then for well-definedness of orbital integrals, finiteness of nilpotent orbits and the validity of the theorem of Shalika (\ref{Shalika}), we need $\text{char}(F)>n$ and \cite[III.4.14]{SS}. However as our $\tG/_E\cong\text{GL}_n$, we can check that \cite[III.4.14]{SS} is valid as long as $\text{char}(F)>n$.\p

Now we discuss the assumption on $\text{char}(k)$. The restriction $\text{char}(k)\not=2$ is used everywhere; we don't bother to deal with quadrics over $\bar{\mathbb{F}}_2$ and wildly ramified group, etc. The only place that we need to assume more is Hypothesis \ref{homo} where we use DeBacker's homogeneity result, whose assumption on $\text{char}(k)$ we don't know how to avoid. However one can do the following: once we establish the result in Section \ref{seccompute} in the case $\text{char}(k)\gg0$, we can compare the result with the method in \cite{Ts}. Roughly speaking, the method in \cite{Ts} computes Shalika germs in terms of the same varieties in Section \ref{geom}, but with (in general) uncontrollable combinatorics.\p

Let's take Theorem \ref{tilded} as an example. The method in \cite{Ts} will compute $\Gamma_{\til{N}_m}(\til{T})$ also in terms of $\#\widetilde{\Sym}^{m'}(C_T)(k)$, $m'\le m$, but with unknown coefficients $P_{m,m'}(q,g)\in\Q(q)[g]$ that are polynomial in the genus $g$ and rational in $q$, independent of the choice of $F$, $k$ and $n=2g+1$. Given that we already know Theorem \ref{tilded} for $\text{char}(k)\gg0$, we know the method in \cite{Ts} must give us the same result.\p

This reduces the restriction on $\text{char}(k)$ to only the restrictions that we need in \cite{Ts}, which assumes $\text{char}(k)\not=2$ because we have a $\Z/2$-grading on $G$, and assumes $(\text{char}(k),n)=1$ for \cite[Claim 2.4]{Ts}. However, what is actually needed for the latter is an self-dual structure on $\Lg=\mathfrak{gl}_n/_k$, which we do have regardless of $\text{char}(k)$. In fact, even if $\tG=SU_n(E/F)$ and $\Lg=\mathfrak{sl}_n$ we are still good, as one can work with $\Lg^*=\mathfrak{pgl}_n$ for the need of \cite[Claim 2.4]{Ts}. In any case, we can drop the assumption $(\text{char}(k),n)=1$.\p

There is also \cite[Hypothesis 3.1]{Ts} which is only known to be true for general groups assuming $\text{char}(k)$ large. However in our case \cite[Hypothesis 3.1]{Ts} is exactly verified by the bijection between nilpotent orbits in $\tLg(F)$ and nilpotent orbits in $\Lg(1)(k)$ described in the beginning of Section \ref{seccompute}. In conclusion, we can work with any $\text{char}(k)$ odd.\p

\let\oldthebibliography=\thebibliography
\let\endoldthebibliography=\endthebibliography
\renewenvironment{thebibliography}[1]{
  \begin{oldthebibliography}{#1}
    \setlength{\itemsep}{2pt}
  }{
    \end{oldthebibliography}
  }
\bibliographystyle{acm}
\bibliography{biblio}

\end{document}